\documentclass[11pt]{article}
\usepackage{amsfonts,amssymb,amsmath,amscd,latexsym,makeidx}
\usepackage{amsthm}
\usepackage{amsmath}
\usepackage{hyperref}
\usepackage{color}
\usepackage{mathrsfs}

\usepackage{mathtools}

\DeclarePairedDelimiter\floor{\lfloor}{\rfloor}

\newtheorem{lem}{Lemma}[section]
\newtheorem{coro}[lem]{Corollary}

\newtheorem{thm}[lem]{Theorem}
\newtheorem{prop}[lem]{Proposition}
\theoremstyle{remark}
\newtheorem{rem}[lem]{Remark}

\newcommand{\D}{\Delta}
\newcommand{\vp}{\varphi}
\newcommand{\R}{\mathbb{R}}
\newcommand{\ve}{\varepsilon}
\newcommand{\rn}{\mathbb{R}^n}

\newcommand{\N}{\mathbb{N}}
\newcommand{\HH}{\mathcal{H}}
\newcommand{\NN}{\mathcal{N}}

\newcommand{\pa}{\partial}

\DeclareMathOperator{\diam}{diam}

\DeclareMathOperator{\capacity}{Cap}

\newcommand{\s}{\mathcal{S}}

\numberwithin{equation}{section}

\begin{document}

\title{Removability of singularities and superharmonicity for some fractional Laplacian equations}

\author{Weiwei Ao, Mar\'ia del Mar Gonz\'alez,  Ali Hyder, Juncheng Wei}
%

\newcommand{\Addresses}{{
  \bigskip
  \footnotesize

Weiwei Ao,  \textsc{Wuhan University,  Department of Mathematics and Statistics, Wuhan, 430072, PR China}\par\nopagebreak
  \textit{E-mail address}, \texttt{wwao@whu.edu.cn}\\

  Mar\'ia del Mar Gonz\'alez, \textsc{Universidad Aut\'onoma de Madrid. Departamento de Matem\'aticas, Campus de Cantoblanco, 28049 Madrid, Spain}\par\nopagebreak
  \textit{E-mail address}, \texttt{mariamar.gonzalezn@uam.es}\\

Ali Hyder, \textsc{Johns Hopkins University, Department of Mathematics,  21218 Baltimore, MD, USA}\par\nopagebreak
  \textit{E-mail address}, \texttt{ahyder4@jhu.edu}\\

  Juncheng Wei, \textsc{Department of Mathematics, University of British Columbia, V6T1Z2 Vancouver, BC, Canada}\par\nopagebreak
  \textit{E-mail address}, \texttt{jcwei@math.ubc.ca}
}}

\maketitle

\begin{abstract}

We study some qualitative properties (including removable singularities and superharmonicity)  of  non-negative  solutions to
$$ (-\D)^\gamma u=fu^p\quad\text{in }\R^n\setminus\Sigma $$
 which are singular at $\Sigma$. Here $\gamma \in (0, \frac{n}{2})$.  Among other things,  we first prove  that if $\Sigma$ is a compact set in $\R^n$ with Assouad dimension  $\bf d$  (not necessarily an integer), ${\bf d}<n-2\gamma$, and   $ u\in L_\gamma(\R^n)\cap L^p_{loc}({\R^n\setminus\Sigma})$ is a non-negative solution
for some  $$p>\frac{n-\bf d}{n-{\bf d}-2\gamma},$$ then $u\in L^p_{loc}(\R^n)$ and $u$ is a distributional solution in $\mathbb R^n$.  Then  we prove that $ (-\D)^\sigma u >0$ for all $ \sigma \in (0, \gamma)$, if $\Sigma=\emptyset$.

\end{abstract}

\medskip

\section{Introduction and statement of results}

Fix $\gamma\in(0,\frac{n}{2})$.
Set $\Sigma\subset\mathbb R^n$ and consider non-negative  solutions to
\begin{align}\label{singu-13}(-\D)^\gamma u=fu^p\quad\text{in }\R^n\setminus\Sigma \end{align}
that are singular at $\Sigma$. Here $f$ is a measurable function; in addition we suppose that there exists $C>0$ such that
\begin{equation}\label{condition-f}
\frac1C\leq f\leq C\quad\text{in }\R^n.
\end{equation}
 To give a meaning to equation \eqref{singu-13}, we need to assume that $u\in L_\gamma(\R^n)$ and $u^p\in L^1_{loc}(\R^n\setminus\Sigma)$, where we have defined, for $s\in\mathbb R$,
$$L_s(\R^n):=\left\{ u\in L^1_{loc}(\R^n):\int_{\R^n}\frac{|u(x)|}{1+|x|^{n+2s}}\,dx<\infty \right\}.$$
Then  \eqref{singu-13} is to be understood in the following sense:
  \begin{align}\label{singular-equation} \int_{\R^n}u(-\D)^\gamma \vp \,dx=\int_{\R^n} fu^p\vp\, dx\quad \text{for every }\vp\in C_c^\infty(\R^n\setminus\Sigma). \end{align}

For the particular power $p=\frac{n+2\gamma}{n-2\gamma}$,  \eqref{singu-13} is the fractional curvature equation in conformal geometry  \cite{Chang-Gonzalez,Case-Chang,Gonzalez:survey}. More precisely, let $|dx|^2$ be the Euclidean metric and consider a conformal change $g=u^{\frac{4}{n-2\gamma}}|dx|^2$ for some smooth positive  function $u$.  One can define the conformal fractional Laplacian operator with respect to the metric $g$, which satisfies
$$P^g_\gamma=u^{-\frac{n+2\gamma}{n-2\gamma}} (-\Delta)^{\gamma}(u \,\cdot). $$
The fractional curvature of $g$ is given by
\begin{equation}\label{fractional-curvature}
Q_{\gamma}^g:=P_\gamma^g(1)=u^{-\frac{n+2\gamma}{n-2\gamma}}(-\Delta)^\gamma u.
\end{equation}
This definition can be extended to more general classes of manifolds  but let us concentrate on Euclidean background. Note here that in the local case $\gamma=1$, curvature \eqref{fractional-curvature} is simply the scalar curvature times a multiplicative constant, while for $\gamma=2$, it coincides with the  $Q$-curvature associated to the Paneitz operator.

When $f\equiv 1$, \eqref{singu-13} yields a fractional order generalization of the Yamabe problem.
In the smooth manifold case some references are \cite{Gonzalez-Qing,Gonzalez-Wang,Kim-Musso-Wei,Mayer-Ndiaye}. Nevertheless, the fractional Yamabe problem in the presence of singularities it is far from being resolved, and the dimension of the singularity is strongly tied to the sign of the curvature \cite{Gonzalez-Mazzeo-Sire}. We will restrict ourselves to the (more interesting) positive case. In particular, isolated singularites have been considered in \cite{CaffarelliJinSireXiong,DelaTorre-Gonzalez,DelaTorre-delPino-Gonzalez-Wei,Ao-DelaTorre-Gonzalez-Wei,Jin-Xiong},
while solutions with singular set $\Sigma$ a smooth submanifold were studied in \cite{Ao-Chan-DelaTorre-Fontelos-Gonzalez-Wei,Jin-Queiroz-Sire-Xiong}, for instance. See also \cite{Ao-Chan-Gonzalez-Wei} for a construction involving more general singular sets (at the expense of not having a complete metric).

In this paper we would like to show, for $\gamma\in(0,\frac{n}{2})$, that the singularity in \eqref{singu-13} is removable, in the sense that the equation holds on all of $\R^n$, this is,
   $u^p\in L^1_{loc}(\R^n)$  and that the above relation  \eqref{singular-equation} holds for every $\vp\in\s(\R^n)$.    Our arguments do not rely on the well known extension problem for the fractional Laplacian \cite{Caffarelli-Silvestre}; instead, we use an integral characterization.

\begin{thm}\label{thm-3}  Take $\Sigma$ a finite number of points. Let $u\geq 0$ be a non-trivial solution to  \eqref{singu-13}. Assume that
$$p\geq \frac{n}{n-2\gamma},\quad\gamma\in\Big(0,\frac n2\Big),$$
and $f$ satisfies \eqref{condition-f}. 
Then $u^p\in L^1_{loc}(\R^n)$  and $u$ is a distributional solution in $\mathbb R^n$.
\end{thm}

In contrast to Theorem \ref{thm-3}, in the (very) subcritical regime it is not possible to have non-negative distributional solutions in $\mathbb R^n$:

\begin{prop}\label{prop:u-vanish}
If $u$ is a non-negative weak solution to
\begin{align}\label{5} (-\D)^\gamma u=fu^p\quad\text{in }\R^n, \end{align}
for $f$ satisfying \eqref{condition-f} with $1<p<\frac{n}{n-2\gamma}$, then $u\equiv0$.
\end{prop}

Theorem \ref{thm-3} and Proposition \ref{prop:u-vanish} are essentially contained in \cite{Chen-Quaas}, where the authors proved that nonnegative classical solutions to the Dirichlet problem for $(-\Delta)^\gamma u=u^p$ in $\Omega\setminus\{0\}$
are weak solutions in $\mathbb R^n$ for $p\geq \frac{n}{n-2\gamma}$. They also classified the asymptotic behavior of the singularity for smaller values of $p$. Nevertheless, our method is very different from theirs and it can be applied to more general singular sets and all powers $\gamma\in(0,\frac{n}{2})$.

\begin{rem}
For the particular value $p=\frac{n}{n-2\gamma}$, solutions with an isolated singularity have been considered in  \cite{Chan-DelaTorre1,Chan-DelaTorre2}, and a complete classification is possible. Note that these have the asymptotic form $1/[ r^{n-2\gamma}(-\log r)^{(n-2\gamma)/{2\gamma}} ]$.
\end{rem}

\begin{thm}\label{thm-4}
Let $\Sigma$ be a $m$-dimensional smooth compact, closed manifold in $\mathbb R^n$, with $0<m<n-2\gamma$, and take
 $u\geq 0$ be a non-trivial solution to  \eqref{singu-13}. Assume that $$p\geq \frac{n-m}{n-m-2\gamma},\quad\gamma\in\Big(0,\frac n2\Big),$$
and $f$ satisfies \eqref{condition-f}. 
Then $u^p\in L^1_{loc}(\R^n)$ and $u$ is a distributional solution in $\mathbb R^n$.
\end{thm}

Now we consider the case of a general compact set $\Sigma$ in $\mathbb R^n$. Although our removability result is stated in terms of its Assouad dimension  $\bf d$, the precise property that we will use is \eqref{measure} on the size of tubular neighborhoods around $\Sigma$. Its relation to the Assouad dimension is proved in \cite{KJV}.  As a consequence, even though our problem is non-local, the main idea in the proof of Theorem \ref{thm-lambda} reduces to finding a particular cutoff function \eqref{def-eta-ve} in a tubular neighborhood of the singular set, which will be controlled by the Assouad dimension.

In paper \cite{KJV} the authors also mention, without proof, the relation between \eqref{measure} and the more standard  Minkowski dimension. Additionally, Assouad dimension has been considered in connection to fractional Hardy inequalities in $\mathbb R^n\setminus \Sigma$ (\cite{Dyda-Vahakangas,Lehrback,Dyda-Ihnatsyeva-Lehrback-Tuominen-Vahakangas}, for instance).

\begin{thm}\label{thm-lambda} Fix $\gamma\in(0,\frac n2)$. Let $\Sigma$ be compact set in $\R^n$ with Assouad dimension  $\bf d$  (not necessarily an integer), ${\bf d}<n-2\gamma$. Assume \eqref{condition-f}, and let  $ u\in L_\gamma(\R^n)\cap L^p_{loc}({\R^n\setminus\Sigma})$ be a non-negative solution to \eqref{singu-13} 
for some  $$p>\frac{n-\bf d}{n-{\bf d}-2\gamma}.$$
 Then $u\in L^p_{loc}(\R^n)$ and $u$ is a distributional solution in $\mathbb R^n$.
\end{thm}

\begin{rem}\label{remark:with-corners}
If the singular set $\Sigma$ is a manifold of dimension $\bf d$ with corners, then we can allow $$p\geq\frac{n-\bf d}{n-{\bf d}-2\gamma}.$$
\end{rem}

We remark that, while Theorem \ref{thm-lambda} contains Theorems \ref{thm-3} and \ref{thm-4}, we have stated them separately since the proof of Theorem \ref{thm-lambda} builds up on the other two. Note also that it is enough to show these results assuming $f\equiv 1$ and we will do so in many places.

Some of the arguments from the proof of Theorem \ref{thm-lambda} are useful in other settings. In particular, they help understanding fractional capacity. Given any compact set $\Sigma\subset \mathbb R^n$, the fractional capacity of order $\gamma$ of $\Sigma$ is defined by
$$\capacity_\gamma(\Sigma):=\inf\left\{ \int_{\R^n}|(-\D)^\frac \gamma 2\vp|^2\,dx:\,\vp\in C_c^\infty(\R^n),\,\vp\geq 1\text{ on }\Sigma \right\}.$$
We give a removability result for $\gamma$-harmonic functions:

\begin{thm}\label{thm:capacity} Set $\gamma\in(0,1)$. Let $h\in L^\infty(\Omega)$ be a solution to
the equation
\begin{equation*}\label{harmonic-eq}
(-\D)^\gamma h=0\quad \text{in } \Omega\setminus\Sigma,
\end{equation*}
 for some compact set $\Sigma\subset\Omega$. If $\capacity_\gamma (\Sigma)=0$ then  $(-\D)^\gamma h=0$ in $\Omega$.
 \end{thm}

Fractional capacity for  the non-linear problem $(-\Delta)^\gamma u=u^{\frac{n+2\gamma}{n-2\gamma}}$ in $\Omega\setminus \Sigma$ was studied in \cite{Jin-Queiroz-Sire-Xiong} for exponents $\gamma\in(0,1)$. Indeed, they provided the asymptotic blow-up rate for positive solutions with a singular set of zero fractional capacity. They also gave an equivalent definition of capacity in terms of the Caffarelli-Silvestre extension for the fractional Laplacian, and considered the relation to the Hausdorff dimension of $\Sigma$.

As a by-product of the arguments in the proof of Theorem \ref{thm:capacity}, we obtain a relation between fractional capacity and property \eqref{measure} that is valid for all $\sigma\in(0,\frac{n}{2})$. This relation could then be rephrased in terms of Assouad or Minkowski dimension (see Proposition \ref{prop:capacity}).\\

In the last part of the paper we show some new superharmonicity properties for the fractional Laplacian. We consider the general problem
\begin{align} \label{eq-1}(-\D)^ \gamma u=F(x)\quad\text{in }\R^n \end{align} with $\gamma\in(0,\frac n2)$ and $F\in L^1_{loc}(\R^n)$. If $\gamma$ is not an integer, we shall assume that $u\in L_\gamma(\R^n)$.
Equation \eqref{eq-1} is to be understood as
\begin{align} \label{eq-2}\int_{\R^n}u(-\D)^\gamma\vp \,dx=\int_{\R^n }F\vp\, dx  \quad\text{for every }\vp\in  C_c^\infty(\R^n).\end{align}

\begin{thm}\label{thm-1} Fix $\gamma\in(0,\frac{n}{2})$. Let $ u \in L_\gamma (\R^n)$ be a   solution to \eqref{eq-1} for some     $F\geq0$ satisfying, in addition, that $u\in L_{s_0}(\R^n)$  for some $s_0\in (0,\frac12)$. Then
\begin{equation*}
(-\D)^\sigma u\geq0\text{ on }\R^n  \quad\text{for every } s_0\leq \sigma<\gamma.\end{equation*}
\end{thm}

 In our second superharmonicity property we do not need to assume any boundedness of $u$ if it is a solution of the semi-linear equation:

\begin{thm}\label{thm-2}
Let $u\in L_\gamma(\R^n)$ be a non-negative distributional solution to \eqref{5} for some $1<p<\infty$ and $\gamma\in(0,\frac n2)$.  Assume also  \eqref{condition-f}.
 Then for every $\sigma\in(0,\gamma)$ we have $$(-\D)^\sigma u\geq 0\quad\text{in }\R^n.$$ \end{thm}

The proofs of the above superharmonicity results will be presented in Section \ref{section:max-principle}, and rely on a bootstrap argument to improve the decay of $u$ at infinity (Section \ref{subsection:bootstrap}). Note that, if $\gamma\in\mathbb N$ we do not need the assumption $u\in L_\gamma(\R^n)$ since we get a better  bound very easily (see Lemma \ref{lemma:integer}), and this gives a new proof in the case of poly-harmonic equations, first proved in Theorem 3.1 of  \cite{WX}.

A source of inspiration for the statement of Theorem \ref{thm-2} is the following pointwise estimate from \cite{Fazly-Wei-Xu}:
\begin{equation*}
-\Delta u\geq \sqrt{\frac{2}{p+1-c_n}}|x|^{\frac{a}{2}}u^{\frac{p+1}{2}}+\frac{2}{n-4}\frac{|\nabla u|^2}{u} \quad\text{in }\mathbb R^n, \quad n>4,
\end{equation*}
for positive bounded solutions of the fourth order H\'enon
equation
\begin{equation*}
(-\Delta)^2u = |x|^au^p \quad\text{in  }\mathbb R^n,
\end{equation*}
for some $a\geq 0$ and $p>1$. This estimate implies, in particular, Theorem \ref{thm-2} for $\gamma=2$ and $\sigma=1$. However, their proof involves an iteration argument in the spirit of Moser, and it is adapted to a local problem, but not generalizable to our non-local equation.

Finally, as a consequence of our removability theorems, we obtain superharmonicity as in Theorem \ref{thm-2} in the presence of singularities:

\begin{coro}\label{cor:max-pple}
Assume that we are in the hypothesis of Theorem \ref{thm-3}, Theorem \ref{thm-4} or Theorem \ref{thm-lambda}. Then
$$(-\D)^\sigma u>0\quad \text{in } \R^n\quad \text{for every}\quad \sigma\in(0,\gamma).$$
\end{coro}

The significance of superharmonicity becomes clear in conformal geometry. Indeed, the sign of the curvature \eqref{fractional-curvature} controls the positivity of the conformal fractional Laplacian operator $P_\gamma$, the location of the first real scattering pole and the geometry and the topology of the manifold \cite{Gonzalez-Mazzeo-Sire,Guillarmou-Qing}. It has been conjectured that, in many cases, positive $Q_\gamma$ curvature implies positive $Q_\sigma$ curvature for $\sigma\in(0,\gamma)$, at least for another metric in the same conformal class. This is precisely the result of \cite{Guillarmou-Qing} for $\gamma=1$ and any $\sigma\in(0,\gamma)$. We also recall \cite{Qing-Raske} and \cite{Zhang} for some related work when $\gamma>1.$

In all these results positivity of the scalar curvature ($\gamma=1$) is the crucial assumption, since it allows to construct a very special comparison function in the proof of superharmonicity. This is precisely the main obstruction to use the same method in other settings. Indeed, this obstruction depends on the local geometry of the manifold and it does not seem to be easily generalizable to the fractional case.  Our Theorem \ref{thm-2}, together with Corollary \ref{cor:max-pple}, hints that it is still reasonable to expect some same kind of property for $\gamma\in(0,1)$ in conformal geometry, at least for a special class of manifolds.

From another point of view, boundary blow up for fractional order equations is reasonably understood. Some references on large solutions are \cite{Abatangelo:large,Abatangelo:very-large,Abatangelo-GomezCastro-Vazquez,Grubb}.\\

Our paper is structured as follows: in Section \ref{section:distributional-solutions} we prove Theorems \ref{thm-3} and \ref{thm-4} for point or smooth singularities. The case of a general singular set $\Sigma$ is considered in Section \ref{section:non-smooth}, where we also give the necessary background on the Assouad dimension. Then, in Section \ref{section:capacity} we use some of these ideas to relate dimension to capacity and give the proof of Theorem \ref{thm:capacity}. The main bootstrap argument comes in Section \ref{section:growth}, which is the main ingredient in the proof of superharmonicity in Section \ref{section:max-principle}. The proof of Proposition \ref{prop:u-vanish} is postponed to this Section since it relies on the previous bootstrap argument. Finally, in Section \ref{Appendix} (the Appendix) we recall some basic facts on the distance function.

\section{Distributional solutions}\label{section:distributional-solutions}

Here we give the proof of Theorems \ref{thm-3} and \ref{thm-4}, when the singular set $\Sigma$ is a smooth manifold of dimension $m\geq 0$.

The fractional Laplacian is defined, for $\sigma\in(0,1)$, by the singular integral formula
\begin{equation*}
(-\Delta)^\sigma u(x)=C_{n,\sigma}P.V.\int_{\mathbb R^n} \frac{u(x)-u(y)}{|x-y|^{n+2\sigma}}\,dy
\end{equation*}
while, for higher powers, say $\sigma=k+\sigma'$, $k\in\mathbb N$, $\sigma'\in(0,1)$,
$$(-\Delta)^\sigma=(-\Delta)^{\sigma'} \circ(-\Delta)^{k}.$$
Note that the meaning of this formula is precise since we are using this on the space $C_c^\infty(\R^n)$,  or for smooth  functions  of the form $\frac{1}{|x|^q}$ on $B_1^c$, where $B_1$ is the unit ball.

Theorem \ref{thm-3} should be compared to the results in \cite{Chen-Quaas}, where the authors show that, for $\gamma\in(0,1)$, any (non-negative) classical solution to
\begin{equation*}
\left\{\begin{split}
(-\Delta)^\gamma u=u^p\quad \text{in }\Omega\setminus\{0\},\\
u=0\quad \text{in }\mathbb R^n\setminus\Omega
\end{split}\right.
\end{equation*}
is a distributional solution of
\begin{equation*}
\left\{\begin{split}
(-\Delta)^\gamma u=u^p+k\delta_0\quad \text{in }\Omega,\\
u=0\quad \text{in }\mathbb R^n\setminus\Omega,
\end{split}\right.
\end{equation*}
for some $k\geq 0$, where $\delta_0$ is the standard Dirac delta. When $p\geq \frac{n}{n-2\gamma}$, the solution extends distributionally to all $\Omega$, this is, $k=0$. In the subcritical case $p\in(1,\frac{n}{n-2\gamma})$  they characterize the asymptotics of  non-removable solutions.

 Theorem \ref{thm-3} is a restatement of the above, but using different ideas in the proof (a very delicate choice of test functions and a dyadic decomposition near the singularity). In particular, our argument contains the core for the generalization to higher dimensional singularities, and also works for any power $\gamma\in(0,\frac{n}{2})$.

\subsection{Point singularities}

 For simplicity let us assume that $\Sigma$ is a single point and $\Sigma=\{0\}$ (after all, the argument is local near each singular point).

We fix two   cut-off functions $\eta_i\in C^\infty(\R^n)$, $i=1,2$, such that $0\leq\eta_i\leq 1$ and  \begin{align*}\eta_1(x)=\left\{\begin{array}{ll}0&\quad\text{for }|x|\leq 1,\\1 &\quad\text{for }|x|\geq 2, \end{array} \right.\qquad  \eta_2 (x):=1-\eta_1(x). \end{align*}
We set
\begin{equation}\label{phieps}
\eta_\ve(x):=\eta_{1,\ve}(x)\eta_2(x), \quad \eta_{1,\ve}(x):=\eta_1\Big(\frac x\ve\Big)\text{ for }\ve>0.
\end{equation}
Let us first estimate $(-\D)^\sigma\eta_\ve$ for this cut-off.


\begin{lem}\label{lemma:estimate} Take $\eta_\ve$ as in \eqref{phieps}.
 For every $\sigma>0$ we have
 \begin{align} |(-\D)^\sigma \eta_\ve(x)|\leq \frac{C_\sigma}{\ve^{2\sigma}}\frac{1}{(1+\frac{|x|}{\ve})^{n+2\sigma}}+\frac{C_\sigma}{(1+|x|)^{n+2\sigma}},
 \quad \text{for all}\quad x\in\mathbb R^n.\label{est4.2} \end{align}
\end{lem}

\begin{proof} The claim follows trivially if $\sigma$ is an integer. For $\sigma\in (0,1)$ we can use the usual formula for $(-\Delta)^\sigma$ in terms of a singular integral to write
\begin{align*}  (-\D )^\sigma &\eta_\ve(x)\\
&=\eta_2(x)(-\D)^\sigma\eta_{1,\ve}(x)+(-\D )^\sigma\eta_2(x)+C_{n,\sigma}PV\int_{\R^n}\frac{(\eta_{1,\ve}(y)-1)(\eta_2(x)-\eta_2(y))}{|x-y|^{n+2\sigma}}\,dy\\&=
\frac{1}{\ve^{2\sigma}} \eta_{2}(x) (-\D)^\sigma\eta_1\Big(\frac x\ve\Big) +(-\D )^\sigma\eta_2(x)+C_{n,\sigma}\int_{\R^n}\frac{(\eta_{1,\ve}(y)-1)(\eta_2(x)-\eta_2(y))}{|x-y|^{n+2\sigma}}\,dy. \end{align*} This gives \eqref{est4.2} for $\sigma\in(0,1)$.\\

When $\sigma>1$ we write $\sigma=k+\sigma_1$ with  $k\in\mathbb{N}$, $0<\sigma_1<1$.  We have \begin{align*}
\D^k\eta_\ve(x)- \D^k\eta_\ve(y)&= [\D^k\eta_{1,\ve}(x)-\D^k\eta_{1,\ve}(y)]\eta_2(x) +[\D^k \eta_2(x)-\D^k\eta_2(y)]\\
&\quad +\Psi_{1,\ve}(x)+\Psi_{2,\ve}(x,y) ,\end{align*} where
\begin{equation*}
\begin{split}
&\Psi_{1,\ve}(x):=\D^k(\eta_{1,\ve}\eta_2)(x)-\eta_2(x)\D^k\eta_{1,\ve}(x)-\D^k\eta_2(x),\\ &\Psi_{2,\ve}(x,y):=\eta_2(x)\D^k\eta_{1,\ve}(y)+\D^k\eta_2(y)-\D^k(\eta_{1,\ve}\eta_2)(y).
\end{split}
\end{equation*}
Notice that   for $\ve<\frac12$ we have $\Psi_{1,\ve}\equiv 0$, and
\begin{align}\label{formula20}
 \Psi_{2,\ve}(x,y)=\left\{\begin{array}{ll}0&\quad\text{for }|x|\leq1,\, y\in\R^n,\\   (\eta_2(x)-1)\D^k\eta_{1,\ve}(y)&\quad\text{for }1\leq|x|\leq 2,\, y\in \R^n,\\
 \D^k[\eta_2(1-\eta_{1,\ve})](y)&\quad\text{for }|x|\geq2,\, y\in\R^n.
 \end{array}\right.\end{align}
Therefore,
$$(-\D)^\sigma\eta_\ve(x)=\eta_2(x)(-\D)^\sigma\eta_{1,\ve}(x)+(-\D)^\sigma\eta_2(x)+I_\ve(x),$$ where $$I_\ve(x)=(-1)^k C_{n,\sigma_1} PV \int_{\R^n}\frac{\Psi_{2,\ve}(x,y)}{|x-y|^{n+2\sigma_1}}\,dy.$$ Since the above integrand has no singularity at $\{x=y\}$ for $\ve<\frac14$ thanks to \eqref{formula20}, one can use integration by parts to deduce that $$|I_\ve(x)|\leq C\frac{\ve^n}{1+|x|^{n+2\sigma}}\quad\text{for every }x\in\R^n,$$
and this concludes the proof of the Lemma.
\end{proof}

\begin{proof}[Proof of Theorem \ref{thm-3}]
 Using the test function $\eta_\ve$ as defined in \eqref{phieps} we get from Lemma \ref{lemma:estimate} that, for a power $p\geq\frac{n}{n-2\gamma}$,
 \begin{equation}\label{formula30}
\begin{split}
 \int _{\R^n}u^p\eta_\ve \,dx&=\int_{\R^n}u(-\D)^\gamma \eta_\ve \,dx
 \\ & \leq C+\frac{C}{\ve^{2\gamma}} \int_{\R^n}\frac{u(x)}{(1+\frac{|x|}{\ve})^{n+2\gamma}}\,dx\\ &=C+\frac{C}{\ve^{2\gamma}}\left( \int_{B_{2\ve}}+\int_{B_1\setminus B_{2\ve}}+ \int_{B_1^c} \right)\frac{u(x)}{(1+\frac{|x|}{\ve})^{n+2\gamma}}\,dx
  \\& \leq  C+\frac{C}{\ve^{2\gamma}}\left( \int_{B_{2\ve}} u\,dx+ \ve^\frac{n}{p'} \left[ \int_{B_1\setminus B_{2\ve}}u^p\,dx\right] ^\frac{1}{p}+ \ve^{n+2\gamma}\right)
  \\& \leq C+\frac{C}{\ve^{2\gamma}}\int_{B_{2\ve}}u\,dx
  +C\left[ \int_{B_1\setminus B_{2\ve}}u^p\,dx\right] ^\frac{1}{p},
 \end{split}
 \end{equation}
 where $1=\frac{1}{p}+\frac{1}{p'}$. Now use that if $a_k^p\leq C_1+C_2a_k$ with $p>1$, then there exists $C_3>0$ such that $a_k^p\leq C_3$, to pass the last term above to the left hand side. We obtain
 \begin{align}  \label{estimate-4.10}  \int _{\R^n}u^p\eta_\ve \,dx& \leq C+\frac{C}{\ve^{2\gamma}}\int_{B_{2\ve}}u\,dx. \end{align}
By H\"older inequality with  $1=\frac{1}{p}+\frac{1}{p'}$ and the previous formula,  we have, for every integer $k\geq -1$,
\begin{align*}
\int_{\{\frac{\ve}{2^{k+1}}<|x|\leq \frac{\ve}{2^{k}} \}}u \,dx
&\leq  C\Big(\frac{\ve}{2^k}\Big)^\frac{n}{p'} \left( \int_{B_1\setminus B_\frac{\ve}{2^{k+1}}} u^p\,dx\right)^\frac1p\\
&\leq C \Big(\frac{\ve}{2^k}\Big)^{\frac{n}{p'} -\frac{2\gamma}{p} }   \left( \int_{B_\frac{\ve}{2^{k+1}}} u\,dx\right)^\frac1p +  C\Big(\frac{\ve}{2^k}\Big)^\frac{n}{p'} \\
&\leq  C \Big(\frac{\ve}{2^k}\Big)^{\frac{n}{p'} -\frac{2\gamma}{p} }   \left( \int_{  B_ \ve} u\,dx\right)^\frac1p+  C\Big(\frac{\ve}{2^k}\Big)^\frac{n}{p'} .\end{align*}
Since $\frac{n}{p'}-\frac{2\gamma}{p}>0$ (iff $p>\frac{n+2\gamma}{n}$, which is true as $p\geq\frac{n}{n-2\gamma}$), summing the above inequality from $k=-1$ to $\infty$, we get $$\int_{B_{2\ve}}u\,dx\leq C\ve^{\frac{n}{p'}-\frac{2\gamma}{p}}  \left( \int_{  B_ \ve} u\,dx\right)^\frac1p+C\ve^\frac{n}{p'}.$$
Using  that $a\leq C_1 a_1+C_2 a_2\leq 2\max\{C_1a_1, C_2a_2\}$ with $a_1\leq a^\frac1p$ implies that $a\leq \max\{(2C_1)^{p'}, C_2a_2\}\leq  (2C_1)^{p'}+ C_2a_2$, we get
$$\int_{B_{2\ve}}u\,dx\leq C\ve^{n-2\gamma\frac{p'}{p}}+C\ve^\frac{n}{p'}.$$ As $ n-2\gamma\frac{p'}{p}\geq 2\gamma$ and $\frac{n}{p'}\geq 2\gamma$, going back to \eqref{estimate-4.10} we get that $u^p\in L^1(B_1)$.

To finish the proof we also need to show that $u$ is a distributional solution on $\R^n$.   Basically we need to show that for every $\vp\in C_c^\infty(\R^n)$,
$$\int_{\R^n} u(-\D)^\gamma (\vp\eta_{1,\ve})\,dx\to \int_{\R^n} u(-\D)^\gamma \vp \,dx .$$
And this follows from the bound $\| (-\D)^\gamma (\vp\eta_{1,\ve})-(-\D)^\gamma\vp\|_{L^{p'}}\leq C$. For this last claim, we need to  estimate  $(-\Delta)^\gamma(\eta_{1,\ve}\varphi)$ as in  Lemma \ref{lemma:estimate}  (see also Lemma \ref{lem-4.4}).


\end{proof}

\subsection{Higher dimensional singular set}

Let $\Sigma$ be a smooth $m$ dimensional compact, closed submanifold of $\R^n$ (or a disjoint union of submanifolds with different dimensions). For $\rho>0$ small we let $\mathcal N_\rho$ to be the geodesic tubular neighborhood of radius $\rho$ around $\Sigma$ and choose Fermi coordinates in $\mathcal N_\rho$ as follows:  first we fix any local coordinate system $y=(y_1,\dots,y_m)$ on $\Sigma$. For every $y_0\in\Sigma$ there exists an orthonormal frame field $E_1,\dots,E_{n-m}$, basis of the normal bundle of $\Sigma$. Set  $N=n-m$. Then we consider the coordinate system $$\Sigma\times\R^{n-m}\ni (y,z)\to y+\sum z_iE_i(y).$$ For $|z|<4\rho$ with $\rho$ small, these generate a  well-defined coordinate system in a neighborhood of $y_0$.
 In this coordinate system the Euclidean metric has the following expansion (\cite{Mazzeo-Smale})
 $$g_{\R^n}=g_{\R^{n-m}}+g_\Sigma +O(|z|)dzdy+O(|z|)dy^2.$$


  We fix  non-negative radially symmetric smooth functions $\eta_1$ and $\eta_2$ in $\mathbb R^{n-m}$ such that \begin{align*}\eta_1(z)=\left\{\begin{array}{ll}   0&\quad\text{for }|z|\leq 1\\ 1& \quad \text{for }|z|\geq 2,\end{array}\right.  \quad \eta_2(z)=\left\{\begin{array}{ll}   1&\quad\text{for }|z|\leq 2\rho\\ 0& \quad \text{for }|z|\geq 3\rho.\end{array}\right.   \end{align*}   For $\ve>0$ small enough we set $$\eta_\ve(x):=\eta_{1,\ve}(z)\eta_2(|z|),\quad \eta_{1,\ve}(z)=\eta_1\Big(\frac{|z|}{\ve}\Big),$$ where $(y,z)\in\Sigma\times \R^{n-m} $ are the Fermi coordinates of $x$.

\begin{lem}\label{lem-4.1}
 We claim that 
 \begin{equation*}\label{inequality1}
 |(-\D)^\sigma\eta_\ve(x)|\leq \frac{C}{\ve^{2\sigma}}\frac{ \eta_2(|z|)}{(1+|z|/\ve)^{N+2\sigma}}+\frac{C}{(1+|x|)^{n+2\sigma}}\quad\text{on }\R^n.
 \end{equation*}
 \end{lem}
\begin{proof} We give a proof only for $\sigma\in (0,1)$.  {The proof for other values of $\sigma$ follows as in the previous section.} 
 We write \begin{align*} (-\D)^\sigma\eta_\ve(x)&=\frac12C_{n,\sigma}\int_{\R^n}\frac{2\eta_\ve(x)-\eta_\ve(x+\tilde x)-\eta_\ve(x-\tilde x)}{|\tilde x|^{n+2\sigma}}\,d\tilde x\\&=: \int_{\R^n}\frac{\Phi_\ve(x,\tilde x)}{|\tilde x|^{n+2\sigma}}\,d\tilde x.
 \end{align*}
For $x\in \R^n\setminus \mathcal N_{4\rho}$ we see that
$$\text{if } \,\,\Phi_\ve(x,\tilde x)\neq0\quad\text{then  } \tilde x\in B_\rho (x)\text{ or } -\tilde x\in B_\rho(x).$$
Therefore, $$|(-\D)^\sigma\eta_\ve(x)|\leq \frac{C}{(1+|x|)^{n+2\sigma}}\quad \text{ for }x\in\R^n\setminus \mathcal N_{4\rho}.$$
For $x\in \mathcal N_{4\rho}$ we write $$(-\D)^\sigma \eta_\ve(x)=\int_{\{|\tilde x|<\rho\}}\frac{\Phi_\ve(x,\tilde x)}{|\tilde x|^{n+2\sigma}}\,d\tilde x+\int_{\{|\tilde x|\geq\rho\}}\frac{\Phi_\ve(x,\tilde x)}{|\tilde x|^{n+2\sigma}}\,d\tilde x=:(I)+(II).$$ Clearly $|(II)|\leq C$.

Let $d=d(x)$ be the distance function from the point $x\in\R^n$ to $\Sigma$. Then, for $x\in \mathcal N_{4\rho}$  we have that $d(x)=|z|$ where $x=(y,z)$. As $\eta_2=1$ for $|z|\leq 2\rho$ (this is, in $\mathcal N_{2\rho}$), for $x\in \mathcal N_\rho$ and  $\tilde x$ small we have the following estimates on  $\Phi_\ve(x,\tilde x)$:
\begin{itemize}
\item[\emph{i})] $\Phi_\ve(x,\tilde x)=0$ \ for \ $4\ve\leq d(x)\leq \rho$ and $|\tilde x|\leq \frac12d(x)$.

To see this use that $d(x\pm \tilde x)\geq d(x)-|\tilde x|\geq 2\ve.$

\item[\emph{ii})]  For $2\ve\leq d(x)\leq \rho$ we have $$\{  \tilde x: \Phi_{\ve}(x,\tilde x)\neq 0 \}=\{  \tilde x: \pm  \tilde x  \in \mathcal N_{2\ve}-x \} =:A_\ve(x),$$
    and we will use this fact later. 

\item[\emph{iii})] For $d(x)\leq 4\ve$, $|\tilde x|\leq \ve  $ we have $$|\Phi_\ve(x,\tilde x)|\leq  |\tilde x|^2\|D^2\eta_\ve\|_{L^\infty}\leq C\frac{|\tilde x|^2}{\ve^2}.$$
    To prove the previous inequality we note that $$\eta_\ve(x)=\eta_1\Big(\frac{d(x)}{\ve}\Big),\quad |\nabla d(x)|=1 , \quad |\nabla^2d(x)|\leq \frac{C}{d(x)},$$ where the last inequality follows from the fact that $d^2$ is smooth in $\mathcal N_{4\rho}$. Therefore, as $\eta_1'(t)=0$ for $|t|\leq 1$, we get  $$D^2\eta_\ve(x)=\frac{1}{\ve^2}\eta_1''\Big(\frac{d(x)}{\ve}\Big)+\frac{1}{\ve}\eta_1'\Big(\frac{d(x)}{\ve}\Big)O(|\nabla ^2d(x)|)=O\Big(\frac{1}{\ve^2}\Big).$$
 \end{itemize}
Now, step  \emph{iii}) yields, for $d(x)\leq 4\ve$,
 \begin{align*} |(I)|\leq C\frac{1}{\ve^2}\int_{\{|\tilde x|<\ve\}}\frac{d\tilde x}{|\tilde x|^{n+2\sigma-2}}  + C\int_{\{|\tilde x|\geq \ve\}}\frac{d\tilde x}{|\tilde x|^{n+2\sigma}}  \leq \frac{C}{\ve^{2\sigma}}.\end{align*}
Finally, using \emph{i})--\emph{ii}) we have, for $4\ve\leq d(x)\leq \rho$,
\begin{align*}
 |(I)|\leq C\int_{\{|\tilde x|\geq\frac12d(x)\}\cap A_\ve(x)} \frac{d\tilde x}{|\tilde x|^{n+2\sigma}}.
 \end{align*}
It remains to estimate the above integral.

In the case that $\mathcal N_{2\ve}$ is  simple type, that is,  of the form  $\{(y,z): |y|< 1,\, |z|<2\ve\}$, then the above integral can be controlled as follows: write $x=(y_0,z_0)$, $\tilde x=(\tilde y, \tilde z)$, so that $|\tilde x|^2=|\tilde y|^2+|\tilde z|^2$, and $\pm \tilde x\in \mathcal N_{2\ve}-x $ is equivalent to $|y_0\pm\tilde y|<1$,  $|z_0\pm \tilde z|<2\ve$, which gives $|\tilde z|\geq \frac12|z_0|$ as $|z_0|\geq 4\ve$. Therefore  (just take the plus sign)
 \begin{align*}
 |(I)|&\leq C \int_{\{|z_0+ \tilde z|<2\ve\}} \int_{\{|y_0+\tilde y|<1\}} \frac{d\tilde yd\tilde z}{(|\tilde y|+|\tilde z|)^{n+2\sigma}} \\
 &\leq  C \int_{\{|z_0+ \tilde z|<2\ve\}} \int_{\{|\tilde y|<2\}} \frac{d\tilde y}{(|\tilde y|+|\tilde z|)^{n+2\sigma}}\,d\tilde z.
 \end{align*}
Making the change of variables $\bar y=\tilde y/|\tilde z|$ in the inside integral,
 \begin{align*}
 |(I)|&\leq C \int_{\{|z_0+ \tilde z|<2\ve\}} \frac{1}{|\tilde z|^{N+2\sigma}}\int_{\R^m}\frac{d\bar y}{(1+|\bar y|)^{n+2\sigma}}\,d\tilde z,
 \\&\leq C \int_{\{|z_0+ \tilde z|<2\ve\}} \frac{1}{|\tilde z|^{N+2\sigma}}\, d\tilde z \\ &\leq C \frac{\ve^N} {|z_0|^{N+2\sigma}},\qquad \text{as } |z_0|\leq 2|\tilde z| \\ &\approx \frac{1}{\ve^{2\sigma}}\frac{1}{(1+\frac{d(x)}{\ve})^{N+2\sigma}}. \end{align*}

If $\mathcal N_{2\epsilon}$ is not of simple type, we proceed as follows. First we cover $\Sigma$ by a finite number of small enough balls and write the metric $g_\Sigma$ in normal coordinates. A neighborhood of $\Sigma\ni q$ is then identified with a  neighborhood in $\mathbb R^{m}\ni 0$ with the metric
$$g_\Sigma=dy^2+O(|y|^2)dy^2.$$
Then we can reduce to the previous type just taking into account the $O(|y|^2)$ error.
 \end{proof}

\begin{proof}[Proof of Theorem \ref{thm-4}]

The proof is very similar to that of Theorem \ref{thm-3}. Here we only give a sketch.

Using the test functions in Lemma \ref{lem-4.1} we have, similarly to \eqref{formula30},
$$\|(-\D)^\gamma\eta_\ve\|^{p'}_{L^{p'}(\R^n)}\leq C+\frac{C}{\ve^{2\gamma p'}}\int_{\mathcal N_{\rho}}\frac{dzdy}{(1+\frac{|z|}{\ve})^{(N+2\gamma)p'}}\leq C+C\ve^{N-2\gamma p'}.$$
Using this one obtains
$$\int_{\R^n}u^p\eta_\ve\,dx\leq C+\frac{C}{\ve^{2\gamma}}\int_{\mathcal N_{2\ve}}u\, dx +C\left(\int_{\mathcal N_{\rho}\setminus \mathcal N_{2\ve}}u^p\,dx\right)^\frac{1}{p}. $$
Hence,
$$\int_{\R^n}u^p\eta_\ve\,dx\leq C+\frac{C}{\ve^{2\gamma}}\int_{\mathcal N_{2\ve}} u \,dx,$$
which is analogous to  \eqref{estimate-4.10}. Since $|\mathcal N_r|\approx r^N$ for $r>0$ small, one can proceed as before, taking a dyadic sequence of distances to $\Sigma$. 
\end{proof}

\section{The non-smooth setting}\label{section:non-smooth}

The Assouad dimension was introduced in \cite{Assouad1,Assouad2} (see also \cite{Luukkainen} for its basic properties). It possesses all the properties any reasonable dimension definition must have. In particular, it is similar to the more standard  Minkowski dimension, but it takes into account all scales (\cite{KJV}).

We will not need the complete definition of Assouad dimension, but just property  \eqref{measure} below for tubular neighborhoods taken from \cite{KJV}. In this paper it is mentioned that the estimate \eqref{measure} also holds in terms of the more usual Minkowski dimension but it is not proved explicitly, so we have decided to keep the original Assouad dimension in our statements.

Its precise definition is as follows: if $(X,d)$ is a doubling metric space, there is a constant $C\geq 1$ such that each ball $B_R(x)$ can be covered by at most $C(r/R)^{-s}$ balls of radius $r$ for all $0<r<R<\diam(X)$, where $s=\log_2 N$. Obviously, this could be true for smaller values of $s$. The infimum of such admissible exponents $s$ is called the upper Assouad dimension of $X$. Considering the restriction metric, this definition extends to all subsets of $X$. The upper Assouad dimension of $E\subset X$ is denoted by $\overline\dim_A(E)$. In the literature, the upper Assouad dimension is usually simply known as the Assouad dimension of $E$, and we will denote by ${\bf d}$.


Now we look at the size of a tubular neighborhood $\mathcal N_r$.
Let $\Sigma$ be a compact set in $\R^n$ which has the following property: For some $\lambda>0$ there exists $C>0$ such that \begin{align} \label{measure}  \HH^{n-1}(\pa\NN_r\cap B)\leq C r^{n-1}\left(\frac{r}{R} \right)^{-\lambda}, \end{align} for every ball $B$ of radius $R\in (0,\diam(\Sigma))$ centered at $\Sigma$ and for every $r\in(0,R)$. Here, $\HH^s$ denotes the $s$-dimensional Hausdorff measure.
It has been shown in \cite{KJV} that \eqref{measure} holds  for every $\lambda$  bigger than the  Assouad dimension of $\Sigma$.  \\

To prove  Theorem \ref{thm-lambda} we would like to reproduce the arguments in the previous section. Nevertheless, since the distance function to $\Sigma$ is not smooth any longer, we cannot use it to construct a cutoff. Instead, we fix a non-negative function $\rho\in C^\infty_c(B_1)$ such that $\int_{\R^n}\rho \,dx=1$. Setting $\rho_\ve(x)=\frac{1}{\ve^n}\rho(\frac x\ve)$, we define
\begin{align}\label{def-eta-ve} \eta_\ve(x):=1-\int_{\NN_{2\ve}}\rho_\ve(x-y)\,dy. \end{align}
Then $\eta_\ve\in C^\infty(\R^n)$ is non-negative, and it satisfies
\begin{equation}\label{eta-eps}\eta_\ve=1\quad\text{on }\NN_{3\ve}^c\quad\text{and }\eta_\ve=0\quad\text{on }\NN_{\ve}.\end{equation}
 Moreover, $$|\nabla^j\eta_\ve|\leq \frac{C}{\ve^j}\quad\text{for }j=1,2,\dots.$$


\begin{lem} \label{lem-4.4}
Assume that \eqref{measure} holds for some $0<\lambda<n$.  Let $\vp\in C_c^\infty(\R^n)$. Then, setting $\vp_\ve=\vp \eta_\ve$ we have  for every $\sigma>0$,
\begin{align*} (-\D)^\sigma \vp_\ve(x)= \eta_\ve(x)(-\D)^\sigma\vp(x)+I_\ve(x),\end{align*}
where $$|I_\ve(x)|\leq C\frac{1}{\ve^{2\sigma}}\frac{\chi_{\NN_1}(x)}{(1+\frac{d(x)}{\ve})^{n+2\sigma-\lambda}}
+C\ve^{n-\lambda}\frac{\chi_{\NN_1^c}(x)}{d(x)^{n+2\sigma}}.$$
Here $\chi_A$ denotes the characteristic function of the set $A$.
\end{lem}

\begin{proof}
Given $\sigma>0$, let $k$ be the integer part of $\sigma$, that is,   $\sigma=k+\sigma'$ with  $\sigma'\in(0,1)$ and $k\in\{0\}\cup\N$.
First we consider the case when $d(x)\geq 10\ve$ (recall \eqref{eta-eps} here). Then we have
\begin{equation}\label{formula40}
\begin{split}
(-\D)^\sigma \vp_\ve(x)&
=C_{n,\sigma}\int_{\R^n}\frac{\eta_\ve(x)[(-\D)^k\vp(x)-(-\D)^k\vp(y)]+(-\D)^k[\vp(y)(1-\eta_\ve(y))]}{|x-y|^{n+2{\sigma'}}}\,dy\\
&= \eta_\ve(x)(-\D)^\sigma \vp(x)+C_{n,\sigma}\int_{\R^n}\frac{ (-\D)^k[\vp(y)(1-\eta_\ve(y))]}{|x-y|^{n+2{\sigma'}}}\,dy \\
&= \eta_\ve(x)(-\D)^\sigma \vp(x)+C\int_{\R^n}\frac{ \vp(y)(1-\eta_\ve(y))}{|x-y|^{n+2{\sigma}}}\,dy  \\
 &=: \eta_\ve(x)(-\D)^\sigma \vp(x)+I_\ve(x),
\end{split}  \end{equation}
  where the second last inequality follows by  integration by parts. Notice that the integrand is not singular at $y=x$ as the function $1-\eta_\ve$ is supported in $\NN_{3\ve}$.

Next we  estimate $I_\ve$ for $10\ve\leq d(x)\leq 1$.
 We have by the co-area  formula (see e.g. \cite[Section 3.4.3]{Evans-Gariepy})
\begin{align}\label{equ-3.5}  | I_{\ve}(x)| &\leq C \int_{\NN_{3\ve}}\frac{dy}{|x-y|^{n+2\sigma}}
= C\int_{0}^{3\ve}\int_{\pa\NN_r} \frac{d\HH^{n-1}(y)}{  |x-y| ^{n+2\sigma}} \,dr.
\end{align}
Note here that the distance function to $\Sigma$ is a 1-Lipschitz function even if $\Sigma$ is very bad. In particular, by Rademacher's theorem, it is differentiable \emph{a.e.} with $|\nabla d|=1$. Although these facts are well known, we provide a short proof in the Appendix.

Let $\tilde x\in \Sigma$ be such that it minimizes the distance of $x$ from $\Sigma$, that is,  $d(x)=|x-\tilde x|=:R$.
Then for $0<r\leq 3\ve$ we have
\begin{align*}\pa\NN_r\subset (\pa\NN_r\cap B_R(\tilde x))\bigcup_{k\geq1}\left\{\pa\NN_r\cap({\overline B_{ 2^kR}(\tilde x)\setminus B_{2^{k-1}R}(\tilde x)})\right\}.\end{align*}
Notice that
\begin{align*} \label{R-dx}|x-y|\geq C 2^k R\quad \text{for }y\in  \pa\NN_r\cap({\overline B_{ 2^kR}(\tilde x)\setminus B_{2^{k-1}R}(\tilde x)}),\quad k\geq 1,  \end{align*}
and also
$$|x-y|\geq C R\quad\text{for }y\in\pa\NN_r\cap \overline B_R(\tilde x).$$
This, and \eqref{measure} imply that \eqref{equ-3.5} can be estimated by
\begin{align*}
 | I_{\ve}(x)| &  \leq C   \sum_{k\geq 0} \frac{1}{(2^k R)^{n+2\sigma}}\int_{0}^{3\ve} r^{n-1} \left(\frac{r}{2^kR}\right)^{-\lambda}\,dr   \\
 &\leq C \frac{\ve^{n-\lambda}} {R^{n+2\sigma-\lambda}}\sum_{k\geq 0} \frac{1}{(2^{n+2\sigma-\lambda})^k}   
 \\&\leq C\frac{\ve^{n-\lambda}}{d(x)^{n+2\sigma-\lambda}}.
  \end{align*}
 The above proof also shows that $|\NN_{3\ve}|\leq C \ve^{n-\lambda}$. Therefore, as $|x-y|\geq C d(x) $ for $d(x)\geq 1$, we easily get that
 $$| I_{\ve}(x)|\leq C  \frac{1}{d(x)^{n+2\sigma}}|\NN_{\ve}|\leq C  \frac{\ve^{n-\lambda}}{d(x)^{n+2\sigma}}\quad\text{for }d(x)\geq 1.$$

 Finally, we treat the case $d(x)\leq 10\ve$. As the term $\eta_\ve(-\D)^\sigma\vp$ is bounded, instead of estimating $I_\ve$ we estimate the term $(-\D)^\sigma\vp_\ve$ in \eqref{formula40}.
  We  have
 \begin{align*} (-\D)^\sigma\vp_\ve(x)=\frac12C_{ {n},\sigma}\int_{\R^n}\frac{2(-\D)^k\vp_\ve(x)-(-\D)^k\vp_\ve (x+h)-(-\D)^k\vp_\ve(x-h)}{|h|^{n+2\sigma'}}\,dh.\end{align*}
 Since the integrant is bounded by $$\frac{C}{\ve^{2k}|h|^{n+2\sigma'}}\min\left\{ \frac{|h|^2}{\ve^2} ,1\right\},$$
 for a constant $C$ depending on $\vp$,
 one easily obtains $$|(-\D)^\sigma\vp_\ve(x)|\leq C\frac{1}{\ve^{2\sigma}}\quad\text{for }d(x)\leq 10\ve.$$
 This concludes the proof of the Lemma.

\end{proof}

\begin{proof}[Proof of Theorem \ref{thm-lambda}]   We choose  $\lambda>\bf d$ but very close to $\bf d$ so that $n-2\gamma p'\geq \lambda$ (equivalently, $p\geq \frac{n-\lambda}{n-\lambda-2\gamma}$).

  We fix $\vp\in C_c^\infty(\R^n)$ such that $\vp\geq 0$ and $\vp\equiv 1$ on $\NN_1$.
 Using the test function $\vp_\ve:=\vp\eta_\ve$   in \eqref{singular-equation}, where $\eta_\ve$ as defined in \eqref{def-eta-ve},  and together with Lemma \ref{lem-4.4} we get
 \begin{equation} \label{formula50}
\begin{split}
 \int _{\R^n}u^p\vp_\ve \,dx&=\int_{\R^n}u(-\D)^\gamma \vp_\ve \,dx
 \\ & \leq C+\frac{C}{\ve^{2\gamma}} \int_{\NN_1}\frac{u(x)}{(1+\frac{d(x)}{\ve})^{n+2\gamma-\lambda}}\,dx\\ &=C+\frac{C}{\ve^{2\gamma}}\left\{ \int_{\NN_{3\ve}}+\int_{\NN_1\setminus \NN_{3\ve}}  \right\}\frac{u(x)}{(1+\frac{d(x)}{\ve})^{n+2\gamma-\lambda}}\,dx.
 \end{split}
 \end{equation}
Set $1=\frac{1}{p}+\frac{1}{p'}$. Note that, again by the co-area formula, and the fact that
 $|\partial\NN_r|\leq C r^{n-\lambda-1}$ for $r\leq 1$ by \eqref{measure},
    \begin{align}   \int_{\NN_1\setminus \NN_{3\ve}}\frac{dx}{(1+\frac{d(x)}{\ve})^{p'(n+2\gamma-\lambda)}} =\int_{3\ve}^1\int_{\partial\NN_r}\frac{d\mathcal H^{n-1}(x)}{(1+\frac{d(x)}{\ve})^{p'(n+2\gamma-\lambda)}} \,dr\leq C \ve^{n-\lambda}.\label{est-Ive-3.8}
   \end{align}
Using  H\"older inequality in the last term in \eqref{formula50} and substituting the above expression, we obtain
 \begin{equation*}\label{estimate-4.6}
 \begin{split}
 \int _{\R^n}u^p\vp_\ve \,dx&\leq  C+\frac{C}{\ve^{2\gamma}}\left\{\int_{\NN_{3\ve}} u\,dx+ \ve^\frac{n-\lambda}{p'} \left( \int_{\NN_1\setminus \NN_{3\ve}}u^p\,dx\right) ^\frac{1}{p} \right\}
  \\& \leq C+\frac{C}{\ve^{2\gamma}}\int_{\NN_{3\ve}}u\,dx +C\left( \int_{\NN_1\setminus \NN_{3\ve}}u^p\,dx\right) ^\frac{1}{p},
 \end{split}
 \end{equation*}
where we have used that  $\lambda\leq n-2\gamma p'$.
As $\vp_{\ve}=1$ on $\NN_1\setminus \NN_{3\ve}$,   we deduce that
 \begin{equation*} \int_{\NN _1\setminus \NN_{3\ve}}u^p\,dx\leq C+ \frac{C}{\ve^{2\gamma}}\int_{\NN_{3\ve}}u\,dx.
 \end{equation*}
 Once  this main estimate has been obtained, proceeding as in the previous subsections one can prove that $u\in L^p_{loc}(\R^n)$.\\

Next we show that $u$ is a distributional solution in $\R^n$:
 for any  $\vp\in C_c^\infty(\R^n)$, taking $\vp_\ve:=\vp\eta_\ve$ as a test function we obtain \begin{align*} \int_{\R^n} u^p\vp_\ve\,dx=\int_{\R^n}u\eta_\ve(-\D)^\gamma \vp\,dx +\int_{\R^n} uI_\ve(x)\,dx, \end{align*} where $I_\ve$ is as in Lemma \ref{lem-4.4} with $\sigma=\gamma$.
 It follows that $$\int_{\R^n} u^p\vp_\ve\,dx\to \int_{\R^n} u^p\vp\,dx\quad\text{and  } \quad \int_{\R^n}u\eta_\ve(-\D)^\gamma \vp\,dx\to \int_{\R^n}u(-\D)^\gamma \vp\,dx$$
 as $\ve\to 0$ thanks to the above bounds.


  Since $\lambda\leq n-2\gamma p'$, from Lemma \ref{lem-4.4} with $\sigma=\gamma$ we get that $\|I_\ve\|_{L^{p'}(\NN_1)}\leq C$ independently of $\ve$ (the proof is similar to \eqref{est-Ive-3.8}, using the co-area formula).  Moreover,  as $u\in L_{\gamma}(\R^n)$, we have $$\lim_{\ve\to0}\int_{\NN_\delta^c}u(x) |I_\ve(x) |\,dx=0\quad\text{for every  }\delta>0.$$
 Hence, for every $\delta>0$,
    \begin{align*}
  \lim_{\ve\to0}\int_{\R^n} u(x)|I_\ve (x)|\,dx=
  \lim_{\ve\to0}  \left(  \int_{\NN_\delta}+\int_{\NN_\delta^c}\right) u(x)| I_\ve (x)|\,dx \leq \|u\|_{L^p(\NN_\delta)}\|I_\ve\|_{L^{p'}(\NN_\delta)},
  \end{align*}
  uniformly in $\delta$. Taking $\delta\to 0$ we obtain
$$   \lim_{\ve\to0}\int_{\R^n} u(x)|I_\ve (x)|\,dx=0.$$
  Thus, $u$ is a distributional solution  in $\R^n$.
\end{proof}

Finally, for Remark \ref{remark:with-corners}, assume that $\Sigma$ is a smooth manifold with corners. Since we are simply using estimate \eqref{measure} and not the full machinery of Assouad dimension, our proof includes this case as well. More generally, if \eqref{measure} holds for some compact set $\Sigma$ and $\lambda>0$, and $u\geq 0 $ is a solution to \eqref{singu-13}  with  $p\geq \frac{n-\lambda}{n-{\lambda}-2\gamma}$, then $u$ is a distributional solution in $\R^n$.

\section{Capacity}\label{section:capacity}

Here we verify Theorem \ref{thm:capacity}. The fractional capacity of order $\gamma$ of $\Sigma\subset\mathbb R^n$ is defined by  $$\capacity_\gamma(\Sigma):=\inf\left\{ \int_{\R^n}|(-\D)^\frac \gamma 2\vp|^2\,dx:\,\vp\in C_c^\infty(\R^n),\,\vp\geq 1\text{ on }\Sigma \right\}.$$
The relation between fractional capacity and Hausdorff dimension was studied in \cite{Jin-Queiroz-Sire-Xiong}, where they provided an equivalent notion of capacity in terms of the extension problem for the fractional Laplacian:

\begin{prop}
[\cite{Jin-Queiroz-Sire-Xiong}]
Assume that $\sigma\in(0,1)$ and let $\Sigma\subset\R^n$ be a compact set.
\begin{itemize}
\item[\emph{i.}] If $\mathcal H^{n-2\sigma}(\Sigma)<\infty$, then $\capacity_\sigma(\Sigma)=0$.
\item[\emph{ii.}] If $\capacity_\sigma(\Sigma)=0$, then $\mathcal H^s(\Sigma)=0$ for $s>n-2\sigma$. In particular, the Hausdorff dimension of $\Sigma$ is less or equal to $n-2\sigma$.
\end{itemize}
\end{prop}

Our arguments from the previous section allow us to extend this result to any $\sigma\in(0,\frac{n}{2})$ in terms of property \eqref{measure} (which can then be related to Minkowski or Assouad dimension). More precisely,

\begin{prop}\label{prop:capacity}
Let $\Sigma$ be a compact set in $\R^n$. Assume that \eqref{measure} holds for some $\lambda\in(0,n)$.   Then, for every  $\sigma\in (0,\frac{n-\lambda}{2}]$, we have $\capacity_\sigma(\Sigma)=0$.
\end{prop}

\begin{proof}
It suffices to prove the Proposition for $\sigma=\frac{n-\lambda}{2}$.   We fix a non-negative function $\rho\in C^\infty_c(B_1)$ such that $\int_{\mathbb R^n }\rho\,dx=1$. Similarly to \eqref{def-eta-ve}, we set $\rho_\ve(x)=\frac{1}{\ve^n}\rho(\frac{x}{\ve})$ and  $$\eta_\ve(x):=\int_{\NN_{2\ve}}\rho_\ve(x-y)\,dy.$$ We claim that  for $0<\ve\leq \delta\leq 1$,
 \begin{align}   \int_{\R^n} \eta_\delta(x)(-\D)^\sigma\eta_\ve(x)\,dx=\int_{\R^n}\eta_\ve(x)(-\D)^\sigma\eta_\delta(x)\,dx=O(1)\Big(\frac{\ve}{\delta}\Big)^{2\sigma}.
\label{est-7.1}\end{align}
Indeed, from Lemma \ref{lem-4.4} we have that
$$|(-\D)^\sigma\eta_\delta(x)|\leq \frac{C}{\delta^{2\sigma}}\frac{1}{(1+\frac{d(x)}{\delta})^{n+2\sigma-\lambda}}\leq \frac{C}{\delta^{2\sigma}}\quad \text{for }d(x)\leq 1,$$
which leads to
\begin{align*}\int_{\R^n}\eta_\ve(x)(-\D)^\sigma\eta_\delta(x)\,dx
=O(\delta^{-2\sigma})\int_{\NN_{3\ve}}\,dx=O(\ve^{n-\lambda}\delta^{-2\sigma}) =O(\ve^{2\sigma}\delta^{-2\sigma}),
\end{align*}
where we have used that the measure of the tubular neighborhood $\mathcal N_{3\ve}$ is of order $\ve^{n-\lambda}=\ve^{2\sigma}$.

For $k\geq 1$  we set  (compare to \cite[Section 4.7.2]{Evans-Gariepy} for the proof in the local case) $$\psi_k:=\frac{1}{S_k}\sum_{\ell=1}^k\frac{\eta_{\ve_\ell}}{\ell}$$
where
$$S_k:=\sum_{\ell=1}^k\frac{1}{\ell},\quad  \ve_{\ell}:=\ell!.  $$
Notice that $\psi_k\in C_c^\infty(\R^n)$, $\psi_k\equiv 1$  on $\mathcal N_{\ve_k}$,
$$\min\left\{ \frac{\ve_\ell}{\ve_{\tilde \ell}},\frac{\ve_{\tilde \ell}}{\ve_\ell} \right\}\leq\min\left\{\frac{1}{\ell},\frac{1}{\tilde\ell}\right\}\quad\text{for }\tilde\ell\neq\ell.$$ Therefore,
\begin{align*}
\int_{\R^n}|(-\D)^\frac \sigma2\psi_k|^2\,dx &=\int_{\R^n} \psi_k(-\D)^\sigma\psi_k\,dx\\&=\frac{1}{S_k^2}  \sum_{\ell=1}^k  \frac{1}{\ell^2}\int_{\R^n}\eta_\ell(-\D)^\sigma\eta_{ \ell}\, dx+\frac{1}{S_k^2}\sum_{\ell\neq\tilde\ell; \,\ell,\tilde \ell=1}^k \frac{1}{\ell\tilde \ell}\int_{\R^n}\eta_\ell(-\D)^\sigma\eta_{\tilde \ell} \,dx.
\end{align*}
We estimate the two terms in the right hand side above by \eqref{est-7.1}, so
\begin{align*}
\int_{\R^n}|(-\D)^\frac \sigma2\psi_k|^2\,dx &\leq  \frac{C}{S_k^2}  \sum_{\ell=1}^k  \frac{1}{\ell^2}  +\frac{C}{S_k^2}\sum_{\ell\neq\tilde\ell; \,\ell,\tilde \ell=1}^k \frac{1}{\ell\tilde \ell} \min\left\{ \frac1\ell,\frac{1}{\tilde\ell}\right\}^{2\sigma} \\&\leq \frac{C}{S_k^2}+\frac{C}{S_k^2}\sum_{\ell=1}^k\frac{1}{\ell}\sum_{\tilde\ell=1}^k\frac{1}{\tilde\ell^{1+2\sigma}}  \\ &\leq \frac{C}{S_k^2}+\frac{C}{S_k}\\&\xrightarrow{k\to\infty}0.
\end{align*}
This concludes the proof of the Proposition.
\end{proof}

For $\sigma\in(0,1)$, in the definition of $\capacity_\sigma(\Sigma)$, one can take the infimum over  the set of functions in $H^\sigma(\R^n)$ which satisfy
$$0\leq\vp\leq 1 \quad \text{and}\quad \vp\equiv1 \text{  in a small neighborhood of }\Sigma. $$
 Indeed, by a density argument, we can replace the space $C_c^\infty(\R^n)$ by $H^\sigma(\R^n)\cap C^0(\R^n)$.  Then from the relation
$$\int_{\R^n}|(-\D)^\frac \sigma2\vp|^2\,dx=c_{n,\sigma}\int_{\R^n\times\R^n}\frac{|\vp(x)-\vp(y)|^2}{|x-y|^{n+2\sigma}}\,dxdy,$$
we see that $$ \||\vp|\|_{H^\sigma(\R^n)}\leq \|\vp\|_{H^\sigma(\R^n)},$$
 and hence we can  assume that $\vp\geq 0$. Finally, for a given $\ve>0$ and $0\leq \vp\in H^\sigma(\R^n)\cap C^0(\R^n)$ with $\vp\geq 1$ on $\Sigma$, we set
 \begin{align*}
\tilde\vp(x):=\left\{  \begin{array}{ll} 1&\text{for }\vp(x)\geq 1-\ve,\\ \frac{\vp(x)}{1-\ve}&\text{for }0\leq \vp(x)\leq 1-\ve.  \end{array}\right.
\end{align*}
It follows that $0\leq\tilde\vp\leq 1$, $\tilde\vp\equiv1$ in a small neighborhood of $\Sigma$, and $$\|\tilde\vp\|_{H^\sigma(\R^n)}\leq (1+C\ve)\|\vp\|_{H^\sigma(\R^n)},$$
as claimed.


\begin{proof}[Proof of Theorem \ref{thm:capacity}] Since $(-\D)^\gamma h=0$ in $\Omega\setminus\Sigma$, $h$ is smooth in $\Omega\setminus\Sigma$. We fix a smooth domain $\tilde\Omega$ with  $\Sigma\Subset\tilde\Omega\Subset\Omega$ so that $h\in C^\infty(\overline{\tilde \Omega}\setminus\Sigma)$.   Let $H$ be the $\gamma$-harmonic function in $\tilde\Omega$  given by the standard Poisson formula  with boundary data $h$. We claim that $h=H$ in $\tilde\Omega$.

To see this, fix smooth domain $\Omega_k$ such that $\Sigma\Subset\Omega_k\Subset\tilde\Omega$ and $d(\Sigma,\partial\Omega_k)\leq\frac1k$.
Let $\vp_k$ be a minimizer of
$$\left\{\int_{\R^n}|(-\D)^\frac\gamma 2\vp|^2\,dx\,: \,\vp\in H^\gamma(\R^n),\quad 0\leq \vp\leq1,\quad \vp\equiv1\text{ on }\Omega_k\right\}.$$
Then $\vp_k$ satisfies
$$(-\D)^\gamma\vp_k=0\quad\text{in }\R^n\setminus\overline\Omega_k,\quad 0\leq\vp_k\leq1,\quad\vp_k\equiv 1\text{ on }\Omega_k.$$
Since the capacity of $\Sigma$ is $0$, we also have that
 $$\int_{\R^n}|(-\D)^\frac\gamma2\vp_k|^2\,dx\to0.$$
 In particular, by Sobolev embedding, $\|\vp_k\|_{L^p(\R^n)}\to0$ for $p=\frac{2n}{n-2s}$. Hence, up to a subsequence, $\vp_k\to0$ almost everywhere in $\R^n$. 

Since $h$ and $H$ are bounded in $\tilde\Omega$, there exists $M>0$ such that $h-H-M\vp_k\leq 0$ in $\Omega_k$.
Then we see that $$(-\D)^\gamma (h-H-M\vp_k)=0\quad\text{in }\tilde\Omega\setminus\bar\Omega_k,\quad h-H-M\vp_k\leq0\quad\text{in }\R^n\setminus(\tilde\Omega\setminus\Omega_k).$$
Thus by maximum principle $$h-H-M\vp_k\leq 0\quad\text{in }\tilde\Omega\setminus\Omega_k.$$ As $\vp_k\to0 \,\,a.e.$, taking the limit $k\to\infty$ we have that $h\leq H$. In a similar way, $h\geq H$.

 \end{proof}

\section{The growth at infinity}\label{section:growth}

\subsection{Preliminary estimates}

We start with some  preliminary bounds:

\begin{lem} \label{lem-decay}Let $\vp_0\in C^2(\R^n)\cap L_\sigma(\R^n)$ for some $\sigma\in(0,1)$.
Then
\begin{equation*}\label{11}\begin{split}
 |(-\D)^\sigma& \vp_0(x)| \\
 &\leq C(n,\sigma)\left(\frac{ \|D^2\varphi_0\|_{L^\infty(A_1(x))}}{(1+|x|)^{2\sigma-2}}+\frac{\|\vp_0\|_{L^\infty(A_2)}}{(1+|x|)^{2\sigma} }+ \left|  \int_{\{1+\frac{|x|}{2}\leq |x-y|\leq 1+ 2|x|\}} \frac{\vp_0(y)}{|x-y|^{n+2\sigma}} \,dy \right|\right), 
\end{split}
\end{equation*}
where $$A_1(x):=B_{1+\frac{|x|}{2}}(x),\quad  A_2(x):=\rn\setminus A_1(x),\quad A_1:=B_{1+\frac{|x|}{2}}(0),\quad  A_2:=\R^n\setminus A_1.$$
   In particular,
    if there exists $\rho>0$ such that  $$\sup_{x\in\R^n}|x|^{\rho+|\alpha|}|D^\alpha\vp_0(x)|<\infty \quad\text{for every multi-index }\alpha \quad\text{with }0\leq |\alpha|\leq 2,$$
then
\begin{align*}|(-\D)^\sigma \vp_0(x)|\leq C(n,\sigma,\vp_0) \left\{    \begin{array}{ll}   ( 1+|x|)^{-2\sigma-\rho} \quad&\text{if }\rho< n \\  ( 1+|x|)^{-2\sigma-\rho}\log(2+|x|) \quad&\text{if }\rho= n\\ ( 1+|x|)^{-2\sigma-n} \quad&\text{if }\rho>n.\end{array}\right.\end{align*}
    \end{lem}

\begin{proof}  The proof is  standard but we give the details for completeness. We shall use the following definition of $(-\D)^\sigma $
$$ (-\D)^\sigma u(x)=\frac{1}{2}C_{n,\sigma}\int_{\R^n}\frac{2\varphi_0(x)-\varphi_0(x+y)-\varphi_0(x-y)}{|y|^{n+2\sigma}}\,dy.$$
Then  we have
 \begin{align*}
 \left|(-\Delta)^{\sigma}\varphi_0(x)\right|&\leq C\left(I_1+I_2 \right),
 \end{align*}
 where we have defined
 $$I_i:=\left|\int_{A_i}\frac{\varphi_0(x+y)+\varphi_0(x-y)-2\varphi_0(x)}{|y|^{n+2\sigma}}\,dy\right|,\quad i=1,2.$$
 Noticing that
 $$|\varphi_0(x+y)+\varphi_0(x-y)-2\varphi_0(x)|\leq \|D^2\varphi_0\|_{L^\infty(A_1(x))}|y|^2\quad\text{for }y\in A_1,$$ we get
 $$I_1\leq \|D^2\varphi_0\|_{L^\infty(A_1(x))}\int_{A_1}\frac{dy}{|y|^{n-2+2\sigma}}\leq C\|D^2\varphi_0\|_{L^\infty(A_1(x))}(1+|x|)^{2-2\sigma}.$$
 On the other hand
 \begin{align*}
  I_2 &\leq 2\left|\int_{A_2}\frac{\varphi_0(x-y)}{|y|^{n+2\sigma}}\,dy\right| + 2|\varphi_0(x)|\int_{A_2}\frac{dy}{|y|^{n+2\sigma}}\leq 2\left|\int_{A_2}\frac{\varphi_0(x-y)}{|y|^{n+2\sigma}}\,dy\right|+
  C|\varphi_0(x)| (1+|x|)^{-2\sigma} \\&=: 2I_3+C|\varphi_0(x)| (1+|x|)^{-2\sigma}.
 \end{align*}
 Now we bound \begin{align*} I_3&\leq  \left| \left( \int_{\{1+\frac{|x|}{2}\leq |y|\leq 1+2|x|\}}+\int_{\{1+2|x|\leq |y|\}}\right)\frac{\varphi_0(x-y)}{|y|^{n+2\sigma}}\,dy\right|    \\  & \leq \left|  \int_{\{1+\frac{|x|}{2}\leq |x-y|\leq 1+2|x|\}} \frac{\vp_0(y)}{|x-y|^{n+2\sigma}} \,dy \right| +C\|\vp_0\|_{L^\infty(A_2)}(1+|x|)^{-2\sigma} \\  &\leq \frac{C}{(1+|x|)^{n+2\sigma}}\int_{\{|y|\leq 1+3|x|\}}|\vp_0(y)|\,dy+C\|\vp_0\|_{L^\infty(A_2)}(1+|x|)^{-2\sigma} .\end{align*}
  Combining these estimates we conclude the proof of the Lemma.
  \end{proof}

\begin{lem}\label{convergence} Let $\vp\in C^\infty(\R^n)$ be such that $\vp(x)=\frac{1}{|x|^\rho}$  on $B_1^c$ for some $\rho>0$. Let  $\eta$ be    a smooth cutoff function  such that $$\eta(x)=1 \quad\text{for } |x|\leq 1\quad\text{ and}\quad\eta(x)=0\quad\text{for }  |x|\geq 2.
$$ Denote by $\eta_\ve(x)=\eta(\ve x)$ and  $\varphi_\ve(x):=\varphi \eta_\ve(x)$.  Then for every $\gamma>0$ we have
$$(-\D)^\gamma\vp_\ve\to (-\D)^\gamma \vp\quad\text{ locally uniformly in }\R^n \text{ as } \ve\to 0.$$ Moreover, there exists $C>0$ (independent of $\ve$) such that  \begin{align*}|(-\D)^\gamma \vp_\ve(x)|\leq C(n,\gamma,\vp) \left\{    \begin{array}{ll}   ( 1+|x|)^{-2\gamma-\rho} \quad&\text{if }\rho<  n \\  ( 1+|x|)^{-2\gamma-\rho}\log(2+|x|) \quad&\text{if }\rho= n \\ ( 1+|x|)^{-2\gamma-n} \quad&\text{if }\rho>n.
\end{array}\right.\end{align*}  \end{lem}
\begin{proof}
We write $\gamma=\gamma_0+\gamma_1$ where $0<\gamma_1<1$ and $\gamma_0\in\mathbb{N}\cup\{0\}$. It follows from Lemma \ref{lem-decay}, applied to  $(-\D)^{\gamma_0}\vp_\ve-(-\D)^{\gamma_0}\vp$, that
 $$(-\D)^\gamma\vp_\ve\to(-\D)^\gamma \vp\quad\text{locally uniformly in }\R^n.$$
To prove the second part of the lemma, first we note that
$$|D^\alpha\vp_\ve(x)|\leq \frac{C(\alpha)}{1+|x|^{ \rho +|\alpha|}}\quad\text{for every multi-index }\alpha,\quad \ve>0.$$ In particular, $\phi:=(-\D)^{\gamma_0}\vp_\ve$ satisfies
\begin{equation*}
\begin{split}
&\sup_{\{|x-y|\leq 1+\frac{|x|}{2}\}}|D^2\phi(y)|
\leq \frac{C}{(1+|x|)^{\rho+2\gamma_0+2}},\quad
\sup_{\{|y|\geq 1+\frac{|x|}{2}\}}|\phi(y)|\leq \frac{C}{(1+|x|)^{ \rho+2\gamma_0}}, \quad \text{and}\\
& \qquad\left|  \int_{\{1+\frac{|x|}{2} \leq|x-y|\leq 1+2|x|\}} \frac{\phi(y)}{|x-y|^{n+2\gamma_1}}\,dy  \right| \leq C
   \left\{    \begin{array}{ll}   ( 1+|x|)^{-2\gamma-\rho} \quad&\text{if }\rho<  n \\  ( 1+|x|)^{-2\gamma-\rho}\log(2+|x|) \quad&\text{if }\rho= n \\ ( 1+|x|)^{-2\gamma-n} \quad&\text{if }\rho>n.
\end{array}\right., 
 \end{split}
\end{equation*} and then the desired estimate follows easily.


 \end{proof}

Now we look at the general equation
\begin{equation}\label{formula60}(-\Delta)^\gamma u=F(x)\quad\text{in }\mathbb R^n,
\end{equation}
understood as in \eqref{eq-2}. The first step to prove Theorem \ref{thm-1} is to  show next that, for any  $\delta\geq 2s_0,$ and outside the origin, $\frac{1}{|x|^{n-2\gamma+\delta}}$ is a good test function in \eqref{eq-2}:

\begin{lem}\label{lemma:claim1}
Let $u$ be as in Theorem \ref{thm-1}. Let $\vp\in C^\infty(\R^n)$ be such that $\vp(x)=\frac{1}{|x|^{n-2\gamma+\delta}}$ on $B_1^c$ for some  $\delta\geq  2s_0$.     
The following identity holds:
\begin{equation*}
\int_{\R^n} u(-\Delta)^\gamma\varphi \,dx=\int_{\R^n} F(x)\varphi(x)\,dx.	
\end{equation*}
 In particular,
 \begin{align}\label{est-f}\int_{\R^n}\frac{F(x)}{1+|x|^{n-2\gamma+\delta}}\,dx<\infty\quad\text{for every }\delta\geq 2s_0 .
\end{align}
\end{lem}

\begin{proof}
By Lemma \ref{convergence} with $\rho=n-2\gamma+\delta$ we have that   \begin{align*} \int_{\R^n}F\vp \,dx=\lim_{\ve\to0}\int_{\R^n}F\vp_\ve\, dx=\lim_{\ve\to0}\int_{\R^n}u(-\D)^\gamma\vp_\ve \,dx=\int_{\R^n}u(-\D)^\gamma\vp \,dx<\infty, \end{align*} thanks to monotone convergence theorem and dominated convergence theorem, and  this completes the proof.\\
\end{proof}

\subsection{A bootstrap argument for the semilinear equation}\label{subsection:bootstrap}

From the discussion in the previous subsection, if $u$ is a solution to
\begin{align}\label{55} (-\D)^\gamma u=fu^p\quad\text{in }\R^n,
\end{align}
a bootstrap argument allows to improve the estimate on the growth of $u$ at infinity. Indeed,

\begin{lem}\label{u^p-estimate}
Let $u\in L_\gamma(\R^n)$ be a non-negative solution to  \eqref{55} for some $1<p<\infty$ and $\gamma\in(0,\frac n2)$.  Assume that there exists $C>0$ such that
\begin{equation*}\label{condition-ff}
\frac1C\leq f\leq C\quad\text{in }\R^n.
\end{equation*}
Then
 \begin{align}\label{equation30}
 \int_{\R^n}\frac{f(x)u^p(x)}{1+|x|^{n-2\gamma}}\,dx<\infty.
 \end{align}

\end{lem}

\begin{proof}
For $\delta>0$ we fix $\vp\in C^\infty(\R^n)$ such that $\vp(x)=\frac{1}{|x|^{n+\delta}}$ on $B_1^c$. Letting  $\eta_\ve$ as before we set  $\vp_\ve:=\eta_\ve\vp$. Then together with  Lemma \ref{convergence}, dominated convergence theorem and monotone convergence theorem we get
$$\infty>\int_{\R^n}u(-\D)^\gamma\vp \,dx=\int_{\R^n}\frac{fu^p(x)}{1+|x|^{n+\delta}}\,dx.$$
Hence, as   $f$ has a positive lower bound,   we get  \begin{equation*} \int_{\R^n}\frac{u^p(x)}{1+|x|^{n+\delta}}\,dx<\infty\quad\text{for every $\delta>0$}.
\end{equation*}
For  any $q>n$ we can write $q=q_1+q_2$ with $q_1p>n$ and $q_2p'>n$. Then by H\"older inequality we get that
\begin{align*}
\int_{\R^n}\frac{u(x)}{1+|x|^q}\,dx\leq C\left(\int_{\R^n}\frac{u^p(x)}{1+|x|^{q_1 p}}\,dx\right)^\frac1p\left(\int_{\R^n}\frac{dx}{1+|x|^{q_2p'}}\right)^\frac{1}{p'}<\infty\quad\text{for every }q>n.
\end{align*}
From this, and  Lemma \ref{convergence} we see that   $\vp\in C^\infty(\R^n)$ with $\vp(x)=\frac{1}{|x|^{n-2\gamma+\delta}}$ on $B_1^c$ with  $\delta>0$  can be used as a test function in \eqref{eq-2}, and consequently we have
$$\int_{\R^n}\frac{u^p(x)}{1+|x|^q}\,dx<\infty\quad\text{for every }q>n-2\gamma.$$
Again by H\"older inequality
\begin{align}\label{induction}\int_{\R^n}\frac{u(x)}{1+|x|^q}\,dx<\infty\quad\text{for every }q>n-\frac{2\gamma}{p}.\end{align}
Now we can take $\vp\in C^{\infty}(\R^n)$ with $\vp(x)= \frac{1}{|x|^{n-2\gamma}}$ on $B_1^c$ as a test function, and we obtain \eqref{equation30}.
\end{proof}

Similar arguments, but with a more complex iteration will yield the triviality of solutions to \eqref{55} for $1<p<\frac{n}{n-2\gamma}$ (see Proposition \ref{prop:u-vanish}). We give the details below in Section \ref{subsection:very-subcritical}.

\subsection{The case $\gamma\in\N$ }

Growth estimates are easy to obtain in this case.

\begin{lem}\label{lemma:integer}
If $u$ is a solution to \eqref{5} for some $p>1$, $\gamma$ an integer in $(0,\frac n2)$, and the right hand side satisfies \eqref{condition-f}, then \begin{align}\label{est-int} \int_{\R^n}\frac{u(x)}{1+|x|^s}\,dx<\infty\quad\text{for every }s>n-2\gamma \frac{p'}{p},\end{align} where $\frac{1}{p'}+\frac{1}{p}=1$.
\end{lem}

 Here the meaning of  equation \eqref{5} is that  $u^p\in L_{loc}^1(\R^n)$ (thus $u\in L^1_{loc}(\R^n)$),  and $u$ is a distributional solution.

 Let us first introduce some notations:  for a smooth function $\vp$, let $\mathscr{A}$ be the set of all derivatives of $\vp$ and their products, that is $$\mathscr{A}:=\left\{ \prod_{i=1}^k D^{\alpha_i} \vp:\, k\in\N,\, \alpha_i\in\N^n \right\}.$$ For $\ell\geq 1$  let  $\mathscr{A}_\ell$ be the vector space (over $\R$)  generated by the elements of $\mathscr{A}$ of order $\ell$, that is, generated by the set $$  \left\{ \prod_{i=1}^k D^{\alpha_i} \vp:\,  \sum_{i=1}^k|\alpha_i|=\ell \right\}.$$

We fix $\vp\in C^\infty(\R^n)$ with $\vp\geq0$. By induction one can show that  for every multi-index $\alpha$ with $|\alpha|=k\in[1,q]$ we have   $$D^\alpha\vp^q=\sum_{\ell=1}^k\vp^{q-\ell}\mathscr{F}_{\alpha,\ell},\quad \mathscr{F}_{\alpha,\ell}\in\mathscr{A}_k.$$
Setting $\vp_R(x):=\vp(\frac xR)$ one gets (use that $\vp^{q_1}\leq C\vp^{q_2}$ for $q_1>q_2$) $$|D^\alpha \vp_R^q(x)|\leq \frac{C}{R^{|\alpha|}} \sum_{\ell=1}^{|\alpha|}\vp_R(x)^{q-\ell}\leq \frac{C}{R^{|\alpha|}} \vp_R(x)^{q-|\alpha|}.$$  \\

\begin{proof}[Proof of Lemma \ref{lemma:integer}]
  We fix a non-negative function $\vp\in C_c^\infty(\R^n)$  such that  $\vp=1$ on $B_1$ and $\vp=0$ on $B_2$. Let $\vp_R$ be as above.  We will use H\"older inequality with $\frac{1}{p}+\frac{1}{p'}=1$. For this, fix $q>2\gamma p'$ an integer (we can also take $q=2\gamma p'$, by a density argument; in that case $\vp_R$ will be $C^{2k}$)  and consider the smooth test function $\vp_R^q$. From the equation we obtain
  \begin{align*} \int_{\R^n}fu^p\vp^q_R\,dx &=\int_{\R^n}u(-\D)^\gamma \vp^q_R\,dx \leq \frac{C}{R^{2\gamma}}\int_{B_{2R}}u\vp_R^{q-2\gamma} \,dx = \frac{C}{R^{2\gamma}}\int_{B_{2R}}u\vp_R^\frac qp \vp_R ^{q-2\gamma-\frac qp} \,dx \\ &\leq \frac{C}{R^{2\gamma}}\int_{B_{2R}}u\vp_R^\frac qp\,dx \leq CR^{\frac{n}{p'}-2\gamma }\left(\int_{\R^n} u^p\vp_R^q\,dx\right)^\frac1p,
  \end{align*}
  which leads to
  $$\int_{B_R} u^p(x)\,dx\leq C R^{n-2\gamma p'}\quad \text{for every }R>0.$$
 (Notice that the above estimate implies $u\equiv0$ for $1<p<\frac{n}{n-2\gamma}$). Then by H\"older inequality $$\int_{B_{2R}\setminus B_R}\frac{u(x)}{1+|x|^s}\,dx\leq CR^{-s+n-2\gamma\frac{p'}{p}}.$$
  Taking a diadic sum one obtains \eqref{est-int}.
\end{proof}


\subsection{The model solution}

We consider the linear problem \eqref{formula60}, for  $\gamma\in(0,\frac{n}{2})$ and  $F\in L^1_{loc}(\mathbb R^n)$. Assume that we are in the hypothesis of Theorem \ref{thm-1}. We start the  proof of this Theorem by constructing a solution $v$ of the equation having the best possible decay.

Let $\Gamma$ be the fundamental solution for the fractional Laplacian, this is,
$$(-\D)^\gamma\Gamma(x)=\delta_0,\quad \Gamma(x)=\Gamma_{n,\gamma}(x)=\frac{c_{n,\gamma}}{|x|^{n-2\gamma}},$$
where $\delta_0$ is the Dirac measure centered at $0$.
Since $F\in L^1(B_1)$, the convolution $\Gamma*F\chi_{B_1}$ is well-defined almost everywhere in $\R^n$. Therefore, up to a translation, we can assume that $F\Gamma\in L^1(B_1)$. Consequently,
by the previous Lemma \ref{lemma:claim1}, 
the following function is well-defined
\begin{equation}\label{old-v}
v(x):=\int_{\R^n}\left( \Gamma(x-y)-\Gamma(y)\right)F(y)\,dy,
\end{equation}
Moreover:

\begin{lem} \label{est-Ls} Let $u$ and $F$ be as in Theorem \ref{thm-1}.
We have $v\in L_s(\R^n)$ for every $s\geq 2s_0$. 
\end{lem}

\begin{proof}
We shall use \eqref{est-f} frequently with $\delta=2s$ or $2s_0$. Calculate
\begin{align}\label{equation10}
\int_{\R^n}&\frac{|v(x)|}{1+|x|^{n+2s}}\,dx\\
&\leq\int_{\R^n}F(y)\left(\int_{\{|x|\leq\frac{|y|}{2}\}}+\int_{\{|x|\geq2{|y|}\}}  +\int_{\{\frac{|y|}{2}<|x|<2|y|\}}\right)\frac{|\Gamma(x-y)-\Gamma(y)|}{1+|x|^{n+2s}}\,dxdy \notag\\
&=:\sum_{i=1}^3I_i.\quad \notag
\end{align}
We notice that   $$|\Gamma(x-y)-\Gamma(y)|\leq C\frac{|x|}{|y|^{n-2\gamma+1}} \qquad\text{for }|x|\leq\frac{|y|}{2},$$
and hence
 \begin{align*}
 I_1 \leq C\int_{\R^n}\frac{F(y)}{|y|^{n-2\gamma+1}}\int_{\{|x| \leq\frac{|y|}{2}\}}\frac{dx}{1+|x|^{n+2s-1}}\,   dy \leq C\int_{\R^n}\frac{F(y)}{|y|^{n-2\gamma+2s_0}} \,dy<\infty. \end{align*}
Next, since $$\int_{\{|x|\geq2{|y|}\}} \frac{|\Gamma(x-y)-\Gamma(y)|}{1+|x|^{n+2s}}\,dx\leq\frac{C}{|y|^{n-2\gamma}}\int_{\{|x|\geq2{|y|}\}} \frac{1}{1+|x|^{n+2s}}\,dx\leq\frac{C}{|y|^{n-2\gamma+2s}},$$
 one also has
 $$I_2<\infty.$$
 Finally, we bound
 \begin{align*}I_3 &\leq \int_{\R^n}F(y)\int_{\{\frac{|y|}{2}\leq |x|\leq2|y|\}}\frac{\Gamma(x-y)}{1+|y|^{n+2s}}\,dxdy
 +C\int_{\R^n}\frac{F(y)}{|y|^{n-2\gamma}}\int_{\{\frac{|y|}{2}\leq |x|\leq 2|y|\}}\frac{1}{1+|x|^{n+2s}}\,dxdy\\&\leq C+C\int_{\R^n}\frac{F(y)}{1+|y|^{n+2s}}\int_{\{|z|<3|y|\}}\frac{dz}{|z|^{n-2\gamma}}
\,dy\\ &\leq C+C\int_{\R^n}\frac{F(y)}{1+|y|^{n-2\gamma+2s}}
\,dy\\
 &<\infty.\end{align*}
 Combining all the above estimates we have that the integral in \eqref{equation10} is finite, as desired.\\
  \end{proof}

Finally, we recall a classification result that will be needed below.

\begin{lem}\label{poly} Let $w\in L_s(\R^n)$ for some $s\geq0$. If $$(-\D)^\sigma w=0\quad\text{in }\R^n,$$ for some $\sigma\geq s$, then $w$ is a polynomial of degree at most $\floor{2s}$, where $\floor{2s}\in\mathbb{N}$ is such that $2s-1\leq \floor{2s}<2s$.   \end{lem}
\begin{proof} See e.g. the proof of  \cite[Lemma 2.4]{H-class}.
\end{proof}

\section{Superharmonicity}\label{section:max-principle}

Now we are ready for the proof of the superharmonicity property, this is, Theorems \ref{thm-1} and \ref{thm-2}.

\begin{proof}[Proof of Theorem \ref{thm-1}]
Take $v$ as defined in \eqref{old-v}. 	
From Lemma \ref{est-Ls} we have that $v\in L_s(\R^n)$ for every $s\geq s_0$, $s\neq 2\gamma$.
Hence, if we define $w=u-v$, then $w\in L_s(\R^n)$ for every $s\geq s_0,\,s\neq 2\gamma$. In addition, since,  $s_0<\frac12$ and $(-\D )^\gamma w=0$ in $\R^n$, we conclude that $w\equiv const$, thanks to Lemma \ref{poly}. Thus $$(-\D)^\sigma u(x)=(-\D )^\sigma v(x)=c(n,\sigma,\gamma)\int_{\R^n}\frac{1}{|x-y|^{n-2\gamma+2\sigma}}F(y)\,dy\geq 0,$$
as desired.
\end{proof}

\begin{proof}[Proof of Theorem \ref{thm-2}]
The main idea here is to use Lemma \ref{u^p-estimate} to improve the growth of $u$ at infinity. Thus we can simply use the auxiliary function
\begin{align}\label{new-v}
 v(x):=c_n\int_{\R^N}\frac{1}{|x-y|^{n-2\gamma}}f(y)u^p(y)\,dy
  \end{align}
instead of the old \eqref{old-v}. The rest of the proof is similar to that of Theorem \ref{thm-1}.  \end{proof}

\subsection{Proof of Proposition \ref{prop:u-vanish}}\label{subsection:very-subcritical}
On the one hand, we claim that every non-negative solution to  \eqref{5} with $1<p<\frac{n}{n-2\gamma}$  satisfies  \begin{align}\label{claim2.6}
\int_{\R^n}\frac{u(x)}{1+|x|^q}\,dx<\infty \quad\text{ for every  }q>\frac{n}{p'},
\end{align}
where $1=\frac{1}{p}+\frac{1}{p'}$.

 On the other hand, one has $u=v$, where  $v$ is given by \eqref{new-v}. To see this, recall that we have proved above that $u=v+const$.  In order to justify that this constant vanishes, it is enough to show that $u,v\in L_\delta$ for some $\delta<0$ small. Estimate \eqref{claim2.6} yields the result for $u$. Moreover, to check that  $v\in L_\delta$ for some $\delta<0$ one can use \eqref{est-up-6.4}, and  proceed as in the proof of Lemma \ref{est-Ls}. In fact, one would get that $$\int_{\R^n}\frac{v(x)}{1+|x|^q}\,dx<\infty\quad\text{for every }q>2\gamma.$$
 If $u$ is non-trivial,  then we can find $R>0$ such that  $\int_{B_R}fu^p\,dx>0$. Therefore, as $|x-y|\approx |x|$ for $(x,y)\in B_{2R}^c\times B_R$, we obtain
  \begin{equation*}\label{u-lower} u(x)=c_n\int_{\R^n}\frac{f(y)u^p(y)}{|x-y|^{n-2\gamma}} \,dy \geq  c_n\int_{B_R}\frac{f(y)u^p(y)}{|x-y|^{n-2\gamma}} \,dy  \geq C \frac{1}{|x|^{n-2\gamma}}\quad \text{for } |x|\geq 2R.
  \end{equation*}
 This contradicts \eqref{claim2.6} as $\frac{n}{p'}<2\gamma$ for  $1<p<\frac{n}{n-2\gamma}$.\\

It remains to prove the claim \eqref{claim2.6}, and we do that by an induction argument.  Setting $$s_m=s_m(p):=\sum_{k=1}^m\frac{1}{p^k}$$ we see that  $s_\infty(p)=\frac{1}{p-1}$ and  $$n-2\gamma-2\gamma s_\infty(p)=0\quad \text{for }p=\frac{n}{n-2\gamma}.$$ Therefore, as $s_\infty(p)$ is monotone decreasing in $p\in (1,\infty)$, we have that  $$n-2\gamma-2\gamma s_\infty(p)<0\quad \text{for }1<p<\frac{n}{n-2\gamma}.$$ In particular,  there exists an integer  $m_0\geq 1$ (depending on $p$) such that $n-2\gamma-2\gamma s_{m_0}(p)\leq 0$. We shall take $m_0$ to be the smallest one.

Next we show that \begin{align}\label{m0}\int_{\R^n}\frac{u(x)}{1+|x|^q}\,dx<\infty\quad\text{for every }q>n-2\gamma s_{m},\end{align} with $m=m_0$.  As \eqref{m0} holds for $m=1$,   thanks to \eqref{induction}, we only need to consider the case $m_0>1$.  Let us show that if \eqref{m0}   holds for some $m=m_1\in \{1,\dots, m_0-1\}$, then it also holds for $m=m_1+1$.  We fix a   test function $\vp\in C^\infty(\R^n)$ with $\vp=\frac{1}{|x|^{ \rho}}$ on $B_1^c$,      $\rho:=n-2\gamma-2\gamma s_{m_1} +\delta$, $\delta>0$.
Since $\rho>0$, by Lemma \ref{convergence}, monotone convergence theorem and dominated convergence theorem we get that
$$\int_{\R^n}fu^p \vp\, dx=\int_{\R^n}u(-\D)^\gamma\vp\, dx\leq C \int_{\R^n}\frac{u(x)}{1+|x|^{n-2\gamma s_{m_1}+\delta}}\,dx<\infty,$$
for every $\delta>0$.  By H\"older inequality we conclude that \eqref{m0} holds with $m=m_1+1$.  This proves that \eqref{m0} holds with $m=m_0$.

From the definition of $m_0$ we have that the integral in \eqref{m0} is finite for every $q>2\gamma$. Therefore,  we can take $\vp\in C^\infty(\R^n)$ with $\vp=\frac{1}{|x|^\delta}$ on $B_1^c$, $\delta>0$ as a test function to conclude that
 \begin{align}\label{est-up-6.4}\int_{\R^n}\frac{fu^p(x)}{1+|x|^\delta}\,dx<\infty\quad\text{for every }\delta>0.\end{align}   This, together with H\"older inequality yields \eqref{claim2.6}.

\qed\\

\section{Appendix: distance function from a set}\label{Appendix}

Let $\Sigma$ be a compact set in $\mathbb R^n$, and define the distance function
\begin{equation*}
d(x)=\inf_{y\in\Sigma} \{|x-y|\}.
\end{equation*}

\begin{lem}
The function $d$ is $1$-Lipschitz continuous.
\end{lem}

\begin{proof}
Let $x,y\in\R^n$ be any two points. Then, for every $z\in\Sigma$,
 $$d(x)\leq |x-z|\leq|x-y|+|y-z|.$$ Hence, taking infimum over $z\in\Sigma$ we deduce that $$d(x)-d(y)\leq|x-y|. $$
  Similarly, one sees that $d(y)-d(x)\leq|y-x|$.
\end{proof}

By Rademacher theorem, $d$ is differentiable \emph{a.e.}, and $|\nabla d|\leq 1$ \emph{a.e.}. Next we show that:

\begin{prop}
If $d$ is differentiable at $x\in\R^n\setminus\Sigma$, then $|\nabla d(x)|=1$. Moreover, there exists unique $\bar x\in\Sigma$ such that $d(x)=|x-\bar x|$ and $$\nabla d(x)=\frac{x-\bar x}{|x-\bar x|}.$$
\end{prop}

\begin{proof}

For simplicity, we assume that $d$ is differentiable at $x=0$ and $0\not\in\Sigma$. We can also assume that the distance $d(0)=:r$ is minimized by the point $x_1:=re_1\in\Sigma$. Then, we have that $B_r\cap\Sigma=\emptyset.$ For $0<t<1$, one has $d(tx_1)\leq  |tx_1-x_1|= (1-t)r$. As $B_r\cap \Sigma=\emptyset$, it follows that $d(tx_1)=(1-t)r$ for $0<t<1$. Therefore, $$\nabla d(0)\cdot e_1=\lim_{t\to0^+}\frac{d(0+te_1)-d(0)}{t}=-1. $$ In particular, as $|\nabla d (0)|\leq 1$, we have that $|\nabla d(0)|=1$.  Thus, $\nabla d(0) =-e_1$.

For any  $\bar x\in\Sigma$ with  $d(0)=|\bar x|=r$, as before we get that $\nabla d(0)=-\frac{\bar x}{r}$.  Hence, the minimizer is unique.



\end{proof}

\noindent\textbf{Acknowledgements.} The authors would like to acknowledge the valuable comments from the referee. W. Ao is supported by NSFC of China No.11801421. M.d.M. Gonz\'alez is supported by the Spanish government grant  MTM2017-85757-P. A. Hyder is supported by the Swiss National Science Foundation grant no. P400P2-183866. J. Wei  is  partially supported by NSERC of Canada.

\Addresses

\begin{thebibliography}{10}
\small

\bibitem{Abatangelo:large}
N. Abatangelo.
Large $S$-harmonic functions and boundary blow-up solutions for the fractional Laplacian.
{\em Discrete Contin. Dyn. Syst.} 35 (2015), no. 12, 5555--5607.

\bibitem{Abatangelo:very-large}
N. Abatangelo.
Very large solutions for the fractional Laplacian: towards a fractional Keller-Osserman condition.
{\em Adv. Nonlinear Anal.} 6 (2017), no. 4, 383--405.

\bibitem{Abatangelo-GomezCastro-Vazquez}
N. Abatangelo, D. G\'omez-Castro, J. L. V\'azquez. Singular boundary behaviour and large solutions for fractional elliptic equations. Preprint.


\bibitem{Ao-DelaTorre-Gonzalez-Wei} W. Ao, A. DelaTorre, M. Gonz\'alez and J. Wei. A gluing approach for the fractional Yamabe problem with prescribed isolated singularities. To appear in {\em Journal f\"ur die reine und angewandte Mathematik}.


\bibitem{Ao-Chan-Gonzalez-Wei}
W. Ao, H. Chan, M. Gonz\'alez, J. Wei. Existence of positive weak solutions for fractional Lane-Emden equations with prescribed singular set. {\em Calc. Var. Partial Differential Equations} 57 (2018), no. 6, Art. 149, 25 pp.

\bibitem{Ao-Chan-DelaTorre-Fontelos-Gonzalez-Wei}
W.~Ao, H.~Chan,  A.~DelaTorre, M.~Fontelos, M.d.M.~Gonzalez, J.~Wei.
\newblock On higher dimensional singularities for the fractional Yamabe problem: a non-local Mazzeo-Pacard program.
\newblock {\em Duke Math Journal} 168, n. 17 (2019), 3297--3411.


\bibitem{Assouad1}
P. Assouad. {\em Espaces m\'etriques, plongements, facteurs.}  Th\`ese de doctorat. Publications Mathé\'ematiques d'Orsay, No. 223--7769. U.E.R. Math\'ematique, Universit\'e Paris XI, Orsay, 1977.


\bibitem{Assouad2}
P. Assouad. \'Etude d'une dimension m\'etrique li\'ee \`a la possibilit\'e de plongements dans $\mathbf R^n$. {\em C. R. Acad. Sci. Paris S\'er. A--B} 288 (1979), no. 15, A731--A734.

\bibitem{CaffarelliJinSireXiong}
L.~Caffarelli, T.~Jin, Y.~Sire, J.~Xiong.
\newblock Local analysis of solutions of fractional semi-linear elliptic
  equations with isolated singularities.
\newblock {\em Arch. Ration. Mech. Anal.}, 213 (2014), n. 1, 245--268.

\bibitem{Caffarelli-Silvestre}
L. Caffarelli, L. Silvestre. An extension problem related to the fractional Laplacian. {\em Comm. Partial Differential Equations} 32 (2007) 1245--1260.

\bibitem{Case-Chang}
J.~Case, S.-Y. A.~Chang.
\newblock On fractional {G}{J}{M}{S} operators.
\newblock {\em  Communications on Pure and Applied Mathematics} 69 (2016), no. 6, 1017--1061.

\bibitem{Chang-Gonzalez}
S.-Y. A.~Chang, M.~Gonz{\'a}lez.
\newblock Fractional {L}aplacian in conformal geometry.
\newblock {\em Adv. Math.}, 226 (2011), no. 2, 1410--1432.

\bibitem{Chan-DelaTorre1}
H. Chan, A. DelaTorre. Singular solutions and a critical Yamabe problem revisited. 	Preprint arXiv:1912.10352

\bibitem{Chan-DelaTorre2}
H. Chan, A. DelaTorre. Singular solutions for a critical fractional Yamabe problem. In preparation.

\bibitem{Chen-Quaas}
H. Chen, A. Quaas.
Classification of isolated singularities of nonnegative solutions to fractional semi-linear elliptic equations and the existence results.
{\em J. Lond. Math. Soc.} (2) 97 (2018), no. 2, 196--221.

\bibitem{DelaTorre-Gonzalez}
A.~DelaTorre, M.d.M.~Gonz{\'a}lez.
\newblock Isolated singularities for a semilinear equation for the fractional
  {L}aplacian arising in conformal geometry.
\newblock {\em Rev. Mat. Iber.} 34 (2018), no. 4, 1645--1678.

\bibitem{DelaTorre-delPino-Gonzalez-Wei}
A. DelaTorre, M. del Pino, M.d.M. Gonz\'alez, J. Wei.  Delaunay-type singular solutions for the fractional {Y}amabe problem.
\newblock{\em Math Annalen} 369 (2017) 597--62.


\bibitem{Dyda-Ihnatsyeva-Lehrback-Tuominen-Vahakangas}
B. Dyda, L. Ihnatsyeva, J. Lehrb\"ack, H. Tuominen, A. V\"ah\"akangas.
Muckenhoupt $A_p$-properties of distance functions and applications to Hardy-Sobolev-type inequalities.
{\em Potential Anal}. 50 (2019), no. 1, 83--105.

\bibitem{Dyda-Vahakangas}
B. Dyda, A. V\"ah\"akangas.
A framework for fractional Hardy inequalities.
{\em Ann. Acad. Sci. Fenn. Math.} 39 (2014), no. 2, 675--689.

\bibitem{Evans-Gariepy}  L. Evans,  R.  Gariepy.
Measure theory and fine properties of functions.
\emph{Studies in Advanced Mathematics. CRC Press, Boca Raton, FL}, 1992. viii+268 pp. ISBN: 0-8493-7157-0


\bibitem{Fazly-Wei-Xu}
M. Fazly,  J. Wei, X. Xu.  A point-wise inequality for the fourth order Lane-Emden equation. {\em Analysis \& PDE} 8(2015), no. 7, 1541--1563.


\bibitem{Gonzalez:survey}
M.~d.~M. Gonz{\'a}lez.
Recent progress on the fractional Laplacian in conformal geometry.
Chapter in {\em Recent Developments in Nonlocal Theory}. Berlin, Boston: Sciendo Migration.  G. Palatucci \& T. Kuusi (Eds.) (2018)

\bibitem{Gonzalez-Mazzeo-Sire}
M.~d.~M. Gonz{\'a}lez, R.~Mazzeo, Y.~Sire.
\newblock Singular solutions of fractional order conformal {L}aplacians.
\newblock {\em J. Geom. Anal.}, 22 (2012), no. 3, 845--863.


\bibitem{Gonzalez-Qing}
M.~d.~M. Gonz{\'a}lez, J.~Qing.
\newblock Fractional conformal {L}aplacians and fractional {Y}amabe problems.
\newblock {\em Anal. PDE}, 6 (2013), no. 7, 1535--1576.

\bibitem{Gonzalez-Wang}
M.~d.~M. Gonz{\'a}lez, M.~Wang.
\newblock Further results on the fractional Yamabe problem: the umbilic case.
\newblock {\em J. Geom. Anal.} 28 (2018), no. 1, 22--60.


\bibitem{Grubb}
G.  Grubb. Fractional-order operators: boundary problems, heat equations. {\em Mathematical analysis and applications—plenary lectures}, 51--81, {\em Springer Proc. Math. Stat.}, 262, Springer, Cham, 2018.

\bibitem{Guillarmou-Qing}
C. Guillarmou, J. Qing.
Spectral characterization of Poincar\'e-Einstein manifolds with infinity of positive Yamabe type.
{\em Int. Math. Res. Not. IMRN} 2010, no. 9, 1720--1740.

\bibitem{H-class}
A. Hyder. Structure of conformal metrics on $\R^n$ with constant $Q$-curvature. {\em Differential and Integral
Equations} 32 (2019), no. 7--8, 423--454.



\bibitem{Jin-Queiroz-Sire-Xiong}
T. Jin, O. de Queiroz, Y. Sire, J. Xiong.
On local behavior of singular positive solutions to non-local elliptic equations. {\em Calc. Var. PDE.} 56 (2017), no. 1, Art. 9, 25 pp.


\bibitem{Jin-Xiong}
T. Jin, J. Xiong. Asymptotic symmetry and local behavior of solutions of higher order conformally invariant equations with isolated singularities. Preprint arXiv:1901.01678.

\bibitem{KJV}
A. K\"aenm\"ki, J.  Lehrb\"ack, M.  Vuorinen.
 Dimensions, Whitney covers, and tubular neighborhoods.
 \emph{Indiana Univ. Math. J.} 62 (2013), no. 6, 1861--1889.

\bibitem{Kim-Musso-Wei} S. Kim, M. Musso, J. Wei.  Existence theorems of the fractional Yamabe problem.  {\em Anal. \& PDE} 11 (2018), no. 1, 75--113.


\bibitem{Lehrback}
J. Lehrb\"ack. Hardy inequalities and Assouad dimensions. {\em J. Anal. Math.} 131 (2017), 367--398.

\bibitem{Luukkainen}
J. Luukkainen. Assouad dimension: antifractal metrization, porous sets, and homogeneous measures. {\em J. Korean Math. Soc.} 35 (1998), no. 1, 23--76.

\bibitem{Mayer-Ndiaye} M. Mayer, C. B. Ndiaye. Fractional Yamabe problem on locally flat conformal infinities of Poincare-Einstein manifolds. Preprint  arXiv:1701.05919.

\bibitem{Mazzeo-Smale}
R. Mazzeo, N. Smale. Conformally flat metrics of constant positive scalar curvature on subdomains of the sphere. {\em J. Differential Geom}. 34 (1991), no. 3, 581--621.

\bibitem{Qing-Raske}
J. Qing, D. Raske.
On positive solutions to semilinear conformally invariant equations on locally conformally flat manifolds.
{\em Int. Math. Res. Not.} 2006, Art. ID 94172, 20 pp.


\bibitem{WX} J. Wei,  X. Xu. Classification of solutions of higher order conformally invariant equations. {\em Math. Annalen} 313 (1999), no. 2,  207--228.


\bibitem{Zhang}
R. Zhang.
Nonlocal curvature and topology of locally conformally flat manifolds. {\em Adv. Math.} 335 (2018), 130--169.


\end{thebibliography}
\end{document}